\documentclass[12pt,reqno,a4paper]{amsart}

\usepackage{a4wide}
\usepackage{amsmath}
\usepackage{amsfonts}
\usepackage{amssymb}
\usepackage{amsthm}
\usepackage{changebar}
\usepackage{dsfont}
\usepackage{enumerate}
\usepackage{graphicx}
\usepackage{hyperref}
\usepackage{mathrsfs}
\usepackage{pdfsync}
\usepackage[normalem]{ulem}

\numberwithin{equation}{section}

\newtheorem{thm}{Theorem}[section]
\newtheorem{lem}[thm]{Lemma}
\newtheorem{cor}[thm]{Corollary}
\newtheorem{prop}[thm]{Proposition}
\newtheorem{defin}[thm]{Definition}

\newtheorem{rem}{Remark}

\def\la{{\lambda}}
\def\La{{\Lambda}}

\def\Ga{{\Gamma}}

\def\Om{{\Omega}}

\def\q{Q}
\def\h{F}
\def\E{\mathbb E}
\def\R{\mathbb R}

\newcommand{\ve}{\varepsilon}
\newcommand{\vf}{\varphi}

\newcommand{\D}{\mathfrak D}
\newcommand{\1}{\mathbf 1}

\newcommand{\etabar}{{\hat \eta}}

\newcommand{\Jbar}{{\hat J}}

\newcommand{\exc}{{\mathrm{exc}}}

\newcommand{\LS}{{\mathrm{LS}}}

\begin{document}

\title{Quasi-static limit for the asymmetric simple exclusion}

\author{Anna de Masi}
\address{Anna de Masi\\Universit\`a dell'Aquila, 67100 L'Aquila}
\email{{\tt demasi@univaq.it}}

\author
{Stefano Marchesani}
\address
{Stefano Marchesani: GSSI, 67100 L'Aquila, Italy}
\email
{\tt stefano.marchesani@gssi.it}

\author
{Stefano Olla}
\address
{Stefano Olla: CNRS CEREMADE\\
Universit\'e Paris-Dauphine, PSL Research University\\
75016 Paris, France\\
Institute Universitaire de France and
 GSSI, 67100 L'Aquila, Italy}
\email
{\tt olla@ceremade.dauphine.fr}
\thanks{
This work was partially supported by ANR-15-CE40-0020-01 grant LSD}

\author
{Lu Xu}
\address
{Lu Xu: GSSI, 67100 L'Aquila, Italy }
\email
{\tt lu.xu@gssi.it}

\begin{abstract}
We study the one-dimensional asymmetric
simple exclusion process on the lattice $\{1, \dots,N\}$
with creation/annihilation at the boundaries.
The boundary rates are time dependent and change
on a slow time scale $N^{-a}$ with $a>0$.
We prove that at the time scale $N^{1+a}$ the system evolves
quasi-statically with a macroscopic density profile
given by the entropy solution of the stationary Burgers equation
with boundary densities changing in time,
determined by the corresponding microscopic boundary rates.
We consider two different types of boundary rates:
the ``Liggett boundaries'' that correspond to the projection
of the infinite dynamics, and the reversible boundaries,
that correspond to the contact with particle reservoirs in equilibrium.
The proof is based on the control of the
Lax boundary entropy--entropy flux pairs and a coupling argument.
\end{abstract}

\maketitle
\thispagestyle{empty}

\section{Introduction}
\label{sec:introduction}

The one-dimensional open asymmetric simple exclusion process (ASEP) is one
of the most interesting models in non-equilibrium statistical mechanics, in particular because its stationary state can be explicitly computed
(cf. \cite{Derrida93,Schutz99,Uchi04,Brak06}).
Particles perform asymmetric random walks on the finite lattice $\{1, \dots,N\}$
with the exclusion rule (the jump is suppressed if the site is occupied), and at the
boundaries particles are created and absorbed with given rates.
The dynamics is then characterized by 5
parameters (the asymmetry of the random walks
and the 4 boundary rates).

In a seminal article \cite{Liggett75}, Liggett introduced special boundary rates
such that the corresponding dynamics approximates optimally
the dynamics of the infinite system with different densities at $\pm\infty$.
These boundary conditions correspond to the projection of the infinite dynamics
with respect to the Bernoulli measures with different densities on the left
and on the right of the system. Under this choice of boundary conditions
Bahadoran \cite{Baha12}
proved the hydrodynamic limit for the density profile: under the hyperbolic
space-time scaling the density profile converges to the unique $L^\infty$
entropy solution of the Burgers equation satisfying the Bardos--Leroux--N\'ed\'elec
boundary conditions \cite{BNL} in the sense of Otto \cite{Otto96}.
Hydrodynamic limit for general boundary conditions has been recently proven in \cite{lu21}.
In \cite{Baha12} Bahadoran also proves the hydrostatic limit,
i.e., the macroscopic limit of the stationary profile satisfies the stationary
Burgers equation with same boundary conditions. This is the solution of
the variational problem maximizing the stationary flux in case of a density gradient
with opposite sign to the drift generated by the asymmetry, or minimizing
the stationary flux in the other case. This is consistent with the phase diagram
proved in \cite{Derrida93,Schutz99}. The proof in \cite{Baha12} relies on
an extension of the coupling argument used by Rezakhanlou
in \cite{reza} in the infinite dynamics, and on the particular boundary conditions
which are such that at equal density (balanced case) the stationary measure
is known explicitly (given by the Bernoulli measure at the boundary density).

In this article we study the quasi-static hydrodynamic limit for the open ASEP.
This limit is taken in a time scale that is larger than the typical one where
the system converges to equilibrium.
Changing the boundary condition at this time scale, the system is
globally close to the corresponding stationary state.
Quasi-static evolutions are usually presented as idealization
of real thermodynamic transformations among equilibrium states.
They are necessary concepts in order to construct thermodynamic potentials,
for example, to define the thermodynamic entropy from Carnot cycles.
Here we are interested in the quasi-static evolution among \emph{non-equilibrium
stationary states}. 
These quasi-static hydrodynamic limits have
been already studied in the symmetric simple exclusion as well
as in other diffusive systems \cite{DOQS}.
We are here interested in the asymmetric case
where currents of density do not vanish in the limit.
Since in ASEP the typical time scale of convergence to stationarity is hyperbolic,
we look at larger time scales
changing the boundary rates in this time scale.
Consequently, at each instant of time the system is \emph{close}
to the corresponding stationary state determined by the varying
boundary conditions.
We prove that the density profile converges to the entropy
solution of the \emph{quasi-static Burgers equation}
with the boundary conditions given by some time dependent
functions $\rho_\pm(t)$.
We derive this quasi-static evolution for two types of
boundary rates:
\begin{enumerate}
\item Liggett boundaries:
if $p > \frac 12$ is the probability of jumping to the right in the
bulk of the system,
at a macroscopic time $t$ we choose 
$[p\rho_{-}(t), (1-p) (1-\rho_{-}(t))]$ as
rates of creation and annihilation on the left side,
respectively $[(1-p)\rho_{+}(t), p(1-\rho_{+}(t))]$ on the right side.
\item Reversible boundaries: we choose
$[\lambda_-(t)\rho_{-}(t), \lambda_-(t) (1-\rho_{-}(t))]$
as rates of creation and annihilation on the left side,
$[\lambda_+(t)\rho_{+}(t), \lambda_+(t) (1-\rho_{+}(t))]$
on the right side, by accelerating this boundary rates and the symmetric
part of the exclusion process in the bulk. These rates correspond to
contact with reversible reservoirs of particles at the corresponding densities
$\rho_\pm(t)$, i.e. they separately satisfy detailed balance
with respect to the Bernoulli measure of the corresponding density
for any choice of the functions $\lambda_{\pm}(t)$,
and are independent from the asymmetry bulk parameter $p$.
Even when
$\rho_{-} = \rho_{+}$ and time independent, the stationary probability
distribution is not a product measure in general.
\end{enumerate}

The proof of the result for both cases will proceed as follows.

We first prove it in the balanced but time dependent cases
($\rho_{-}(t)= \rho_{+}(t)$).
Surprisingly this is the most difficult part, and it is proven by
controlling the time average of the \emph{microscopic boundary entropy flux}.
The unbalanced situation is then proven by a coupling argument.

The use of the microscopic entropy production
associated to a Lax entropy--entropy flux pair is already present in
the seminal article of Rezakhanlou \cite{reza}.
J. Fritz and collaborators combined this idea with a stochastic version of
compensated compactness in order to deal with non-attractive dynamics
\cite{Fritz04, FT04}.
Otto \cite{Otto96} introduced the \emph{boundary entropy--entropy flux pairs}
in order to characterize the boundary conditions in the scalar hyperbolic equations.
The main point of this article is to prove that,
in the balanced case when we take the same density on the boundaries,
the time average of the microscopic boundary entropy flux is negligeable
in the quasi-static time scale,
even when boundary conditions change in time (see Propositions
\ref{prop:bd-ent-li} and \ref{prop:bd-ent-ge}).

Notice that these results and the methods are very different
from the symmetric case ($p=\frac 12$) studied in \cite{DOQS}, where
the quasi-static time scale is larger than the diffusive one and
the quasi-static profile
satisfies the Laplace equation with boundary conditions $\rho_\pm(t)$.


\section{ASEP with open boundaries}
\label{sec:asep}

The asymmetric simple exclusion process (ASEP) 
with open boundary conditions is the Markov process on
the configuration space 
\begin{align}
\Om_N := \big\{\eta=(\eta_1,\eta_2,\ldots,\eta_N),\ \eta_i\in\{0,1\}\big\},
\end{align}
with the infinitesimal generator
\begin{equation}
\label{eq:generator}
L_Nf = \la_0 L_\exc f + L_- f + L_+ f, 
\end{equation}
where 
$\la_0>0$, $f$ is any function on $\Om_N$,
$L_\exc$ is the generator of the simple exclusion: 
\begin{equation}
\label{eq:exc}
\begin{split}
&L_{\exc}f :=\sum_{i=1}^{N-1} c_{i,i+1} 
\left[ f\left(\eta^{i,i+1}\right)-f(\eta) \right], \\
&c_{i,j} := p\eta_i(1-\eta_j) + (1-p)\eta_j(1-\eta_i), 
\end{split}
\end{equation}
where $1/2<p\le1$, $\eta^{i,i+1}$
is the configuration obtained from $\eta$ 
upon exchanging $\eta_i$ and $\eta_{i+1}$. 
$L_\pm$ are the generators of 
creation/annihilation processes at the boundaries $i=1$ and $i=N$:
\begin{equation}
\label{eq:bdop}
\begin{split}
L_-f &:= \left[\alpha({1-\eta_1}) + \gamma{\eta_1}\right] 
\big[f(\eta^1)-f(\eta)\big], \\
L_+f &:= \left[\delta({1-\eta_N}) + \beta{\eta_N}\right] 
\big[f(\eta^N)-f(\eta)\big],
\end{split}
\end{equation}
where $\alpha,\gamma,\beta,\delta >0$, $\eta^i$ is the configuration 
obtained from $\eta$ by shifting the status
at site $i$ from $\eta_i$ to $1-\eta_i$. 


\begin{rem}[Stationary states]
\label{rem:ss}
In general the stationary probability distribution has a complicate structure
\cite{Derrida93,Schutz99,Uchi04}.
But for the choice of the boundary rates such that (see \cite[Proposition 2]{Brak06})
\begin{equation}
\label{eq:4}
\bar p := 2p -1 =
\frac{(\alpha+\beta+\gamma+\delta)(\alpha\beta-\gamma\delta)}
{(\alpha+\delta)(\beta+\gamma)},
\end{equation}
the stationary state is given by the Bernoulli product measure with density
\begin{equation}
\label{eq:6}
\rho(\alpha,\beta,\gamma,\delta) =
\frac{\alpha+\delta}{\alpha+\beta+\gamma+\delta}.
\end{equation}
\end{rem}

In this article, we consider time-dependent parameters
$(\alpha,\gamma,\beta,\delta)(t)$.
As in \eqref{eq:generator}, define the Markov generator 
\begin{equation}\label{eq:generator1}
L_{N,t} = \la_0L_\exc + L_{-,t} + L_{+,t}, \quad t\ge0, 
\end{equation}
where $L_{\pm,t}$ are the operators defined by \eqref{eq:bdop}.

We multiply $L_{N,t}$ by $N^{1+a}$ for some $a>0$ and study the macroscopic limit of the corresponding dynamics. 
We now distinguish two cases: the \emph{Liggett boundaries}
where we only speed up the generator by $N^{1+a}$
to the quasi-static time scale, and the
\emph{general reversible boundaries} where there is a further speeding
of the symmetric exclusion and of the boundary rates. 

\subsection{Liggett boundaries}

Take $\la_0=1$, $1/2<p\le1$ and two $\mathcal C^1$ functions 
$\rho_\pm: [0,\infty)\to(0,1)$. 
Choose the parameters in \eqref{eq:generator1} as
\begin{equation}
\label{eq:liggett}
\begin{aligned}
&\alpha(t) := p \rho_-(t), \qquad \gamma(t) := (1-p) (1-\rho_-(t))\\
&\beta(t) := p (1-\rho_+(t)), \qquad \delta(t) := (1-p) \rho_+(t).
\end{aligned}
\end{equation}
Under this choice, if $\rho_+(t) = \rho_-(t) = \rho(t)$, then \eqref{eq:4} is
satisfied and by \eqref{eq:6} we have
\begin{align}
\rho\big(\alpha(t),\beta(t),\gamma(t),\delta(t)\big)= \rho(t).
\end{align}



First introduced by Liggett \cite{Liggett75} with time-independent $\rho_\pm$,
this choice of boundary rates corresponds to the projection
on the finite interval $\{1,2,\ldots,N\}$ of the infinite
ASEP dynamics with Bernoulli distribution
with density $\rho_-$ on the left of $1$, and with density
$\rho_+$ on the right of $N$ (see formulas (5) and (6) in \cite{Baha12}).
This choice allows an effective coupling between the dynamics with open boundaries and the infinite dynamics, which plays a central role in the proof of the corresponding hydrodynamic limit \cite{Baha12}.

\subsection{Reversible boundaries}
\label{sec:reversible-boundaries}

In the Liggett case the boundary rates
are chosen in accordance with $p$.
To deal with more general cases in which the boundary rates are independent of $p$,
and model the contact with \emph{reversible} reservoirs of particles,
we need to speed up the boundary operators. 
We also need to add symmetric exchanges at a higher rate.
Let $\sigma_N$ and $\widetilde\sigma_N$ be two sequences satisfying
\begin{align}\label{eq:assp-ge}
\lim_{N\to\infty} \frac{\sigma_N}N = 0, \quad \lim_{N\to\infty} \widetilde\sigma_N = \infty, \quad \lim_{N\to\infty} \frac{\sigma_N\widetilde\sigma_N}{\sqrt N} = \infty. 
\end{align}
The boundary rates are defined by the choice of two $\mathcal C^1$ functions
$\rho_\pm:[0,\infty)\to(0,1)$ and two $\mathcal C^1$ functions
$\lambda_\pm:[0,\infty)\to \mathbb R_+$, and
\begin{equation}
\label{eq:1}
\begin{split}
&\alpha(t) := \widetilde\sigma_N\lambda_-(t) \rho_-(t), \qquad
\gamma(t) := \widetilde\sigma_N \lambda_-(t) (1-\rho_-(t)),\\
&\beta(t) := \widetilde\sigma_N\lambda_+(t)(1-\rho_+(t)), \qquad
\delta(t) := \widetilde\sigma_N \lambda_+(t)\rho_+(t).
\end{split}
\end{equation}
In order to speed up the symmetric part of the bulk dynamics we
fix $\bar p\in(0,1]$ and we choose
\begin{equation}
\begin{split}
\la_0 = \sigma_N, \quad p = \frac12 + \frac{\bar p}{2\sigma_N},
\end{split}
\end{equation}
The resulting generator of the bulk dynamics is
\begin{equation}
\la_0L_\exc = \bar pL_{\text{tasep}} + \frac{\sigma_N-\bar p}2 L_{\text{ssep}},
\end{equation}
where $L_{\text{ssep}}$ and $L_{\text{tasep}}$ are respectively given by 
\begin{equation}\label{eq:generator2}
\begin{split}
L_{\text{ssep}}f &:= \sum_{i=1}^{N-1} \left[ f\left(\eta^{i,i+1}\right)-f(\eta) \right], \\
L_{\text{tasep}}f &:= \sum_{i=1}^{N-1} \eta_i(1-\eta_{i+1}) 
\left[ f\left(\eta^{i,i+1}\right)-f(\eta) \right].
\end{split}
\end{equation}
This dynamic should not be confused with
the so called \emph{weakly asymmetric exclusion},
since with the choice of the parameters we make in Theorem
\ref{thm:reversible} below the asymmetry is always strong.


\section{Quasi-static evolution}
\label{sec:quasistatic}

\subsection{Quasi-static Burgers equation}
\label{sec:quasi-static-burgers}

For any $\ve>0$, let $\rho^\ve \in L^\infty([0,1]\times\R_+)$ be the \emph{entropy solution} to the initial--boundary problem of the scalar conservation law 
\begin{equation}
\label{eq:2}
\left\{
\begin{aligned}
&\,\varepsilon \partial_t \rho^{\varepsilon}(x,t) + 
\partial_x J(\rho^{\varepsilon}(x,t)) = 0, 
\quad J(\rho)=\bar p\rho(1-\rho),\\
&\,\rho^{\varepsilon}|_{x=0} = \rho_-, 
\quad \rho^{\varepsilon}|_{x=1} =\rho_+, 
\quad \rho|_{t=0} = \rho_0, 
\end{aligned}
\right.
\end{equation}
where $\rho_\pm\in\mathcal C^1(\R_+)$ and $\rho_0 \in L^\infty([0,1])$. 
The definition, existence and uniqueness
of the entropy solution follow from
the work of Otto \cite{Otto96}, by the characterization
of the \emph{boundary entropy--entropy flux pairs} defined below. 

\begin{defin}\label{bLef}
A boundary Lax entropy--entropy flux pair for \eqref{eq:2} is a couple of $\mathcal C^2$ functions $(\h,\q): [0,1]\times\R \to \R^2$ such that 
\begin{align}
J'(u)\partial_u\h(u,w)= \partial_u\q(u,w), \quad \h(w,w)=\q(w,w)=\partial_u\h(w,w)=0
\end{align}
for all $u\in[0,1]$ and $w\in\R$. 
Moreover, we say that the pair $(\h,\q)$ is convex if
$\h(u,w)$ is convex in $u$ for all $w\in\R$. 
\end{defin}

Otto's boundary conditions read in this case as
\begin{equation}
\label{eq:otto}
\begin{split}
& \mathop{\text{esslim}}_{r\to 0^+} \int_0^\infty \q
\big(\rho^{\varepsilon} (r,t),\rho_-(t)\big) \beta(t) dt \le 0, \\
& \mathop{\text{esslim}}_{r\to 0^+} \int_0^\infty \q
\big(\rho^{\varepsilon} (1-r,t),\rho_+(t)\big) \beta(t) dt \ge 0,
\end{split}
\end{equation}
for any boundary entropy flux $\q$ for a convex entropy pair, and
test function $\beta\in \mathcal C_c(\R_+)$ such that $\beta \ge 0$.
In the case of bounded variation solutions, this coincides with the
Bardos--Leroux--N\'ed\'elec boundary conditions \cite{BNL}.

The entropy solution $\rho \in L^\infty([0,1]\times\R_+)$ to the \emph{quasi-static conservation law} 
\begin{align}
\label{eq:qscl}
\partial_x J(\rho(x,t)) = 0, \ x\in(0,1), \quad \rho|_{x=0} = \rho_-, \quad \rho|_{x=1} =\rho_+,
\end{align}
is then defined as the weak-$\star$ limit, for $\ve\to0^+$ of $\rho^\ve$. 
The existence and uniqueness of
this limit is proven in \cite{mox21} as explained below.
Define the \emph{critical line} 
\begin{align}
\label{eq:critical}
\Theta := \{(a,b)\in(0,1)^2; a<1/2, a+b=1\}. 
\end{align}
If $(\rho_-,\rho_+)(t) \notin \Theta$, 
then $\rho$ is constant in $x$ and it is independent
of the initial condition $\rho_0$ in \eqref{eq:2}. 
It is explicitly given by 
\begin{align}
\label{eq:phases}
\rho(x,t) = 
\begin{cases}
\rho_-(t), &\text{if}\ \rho_-(t)<1/2,\:\rho_-(t)+\rho_+(t)<1\ (\text{low density}), \\
\rho_+(t), &\text{if}\ \rho_-(t)>1/2,\:\rho_-(t)+\rho_+(t)>1\ (\text{high density}), \\
1/2, &\text{if}\ \rho_-(t) \ge 1/2,\:\rho_+(t) \le 1/2\ (\text{max current}). 
\end{cases}
\end{align}
Moreover, $\rho$ satisfies the variational conditions
(\cite{Derrida93,Schutz99,Baha12,mox21}): 
\begin{align}\label{eq:variational}
J(\rho(x,t)) = 
\begin{cases}
\sup \{J(\rho); \rho \in [\rho_+(t),\rho_-(t)]\}, 
&\text{if }\rho_-(t) > \rho_+(t), \\
\inf \{J(\rho); \rho\in[\rho_-(t), \rho_+(t)]\}, 
&\text{if } \rho_-(t) \le \rho_+(t). 
\end{cases}
\end{align}
The characterization of the limit $\rho$ is open when $(\rho_-,\rho_+)(t) \in \Theta$.
The variational formula \eqref{eq:variational} remains true,
but $\rho$ may attain two values $\rho_-(t)$ and $\rho_+(t)$. 
For time-independent boundary data $(\rho_-,\rho_+) \in \Theta$, \eqref{eq:qscl} has infinitely many stationary entropy solutions: for each $x^*\in[0,1]$, $\rho_{x^*}(x):=\rho_-\mathbf1_{[0,x^*)}(x)+\rho_+\mathbf1_{[x^*,1]}(x)$ is one of them. 
The choice of $x^*$ in the limit $\ve\to0+$ may rely on the initial data $\rho_0$ in \eqref{eq:2}.
Time-dependent boundary data could further complicate the problem by creating a moving shock.

\subsection{Quasi-static hydrodynamics}
\label{sec:quasi-stat-hydr}

For some $a>0$, denote by $\eta(t)=(\eta_1(t),\ldots,\eta_N(t)) \in \Om_N$ the process generated by $N^{1+a}L_{N,t}$ and some initial distribution $\mu_{N,0}$.
For $i=0$, ..., $N$, let $\chi_{i,N}$ be the indicator function 
\begin{equation}\label{eq:indicator}
\chi_{i,N}(x) :=
\1_{\left\{\left[\frac iN-\frac1{2N}, \frac iN+\frac1{2N}\right) \cap [0,1]\right\}}(x),
\quad x\in[0,1]. 
\end{equation}
For each $N$, define the empirical density $\zeta_N = \zeta_N(x,t)$ as 
\begin{align}
\label{eq:empirical}
\zeta_N(x,t) := \sum_{i=1}^N \chi_{i,N}(x) \eta_i(t), \quad (x,t) \in [0,1]\times[0,\infty). 
\end{align}
Our aim is to show that, as $N \to \infty$, $\zeta_N$ converges to 
the homogeneous density $\rho(t)$ given by \eqref{eq:phases}.
Observe that for reversible boundaries case, the boundary conditions
of the macroscopic quasi-static equation
do not depend on the choice of $\lambda_\pm(t)$.

Denote by $\mathbb P_N$ the distribution on the path space of $\eta(\cdot)$. 
Let $\E_N$ be the corresponding expectation. 
For $\rho\in[0,1]$, denote by $\nu_\rho$
the product Bernoulli measure with density $\rho$.
Given a local function $f$ on $\Om_N$, denote 
\begin{align}
\langle f \rangle(\rho) := \int fd\nu_\rho, \quad \forall\,\rho\in[0,1].
\end{align}

\begin{thm}[Liggett boundaries]
\label{thm:liggett}
Assume that one of the following holds almost everywhere in a time interval $[s,s']$: 
\begin{equation}
\label{eq:assp-t}
\left\{
\begin{aligned}
&\,\rho_-(t)-\rho_+(t)\le0 \\
&\,(\rho_-,\rho_+)(t) \notin \Theta
\end{aligned}
\right.,
\qquad \text{or} \qquad
\left\{
\begin{aligned}
&\,\rho_-(t)-\rho_+(t)\ge0\\
&(\rho_-,\rho_+)(t) \notin \Theta
\end{aligned}
\right.,
\end{equation}
where $\Theta$ is given by \eqref{eq:critical}. 
Suppose further that $a>1/2$, then for any sequence of initial distributions $\{\mu_{N,0};N\ge1\}$, the following convergence holds in probability 
\begin{equation}
\label{eq:limit}
\lim_{N\to\infty} \int_s^{s'} 
\frac1N\sum_{i=0}^{N-\ell} \vf\left(\frac iN,t\right)f\big(\tau_i\eta(t)\big)dt 
= \int_s^{s'}\!\!\!\int_0^1 \vf(x,t) \langle f\rangle\big(\rho(t)\big)dx\,dt
\end{equation}
for any $\ell$, any local function $f$ supported on $\{1,\ldots,\ell\}$
and $\vf \in \mathcal C([0,1]\times\R_+)$,
where $\tau_i$ is the shift operator and $\rho(t)$ is given by \eqref{eq:phases}.
\end{thm}

\begin{rem}
Condition \eqref{eq:assp-t} means that $\rho_-(t)-\rho_+(t)$
keeps its sign for $t\in [s,s']$ and prevents
$(\rho_-(t),\rho_+(t))$ from staying in $\Theta$
for an interval of time of positive measure.
This is necessary when applying the variational
formula and the coupling argument, see Section \ref{sec:proof}.
\end{rem}

\begin{thm}[Reversible boundaries]
\label{thm:reversible}
Assume that one of the conditions in \eqref{eq:assp-t}
holds almost everywhere in $[s,s']$.
In additional, assume \eqref{eq:assp-ge} and 
\begin{align}\label{eq:assp-a-ge}
\lim_{N\to\infty} N^{a-\frac12}\sigma_N = \infty. 
\end{align}
Then, the convergence in \eqref{eq:limit} holds with $\rho(t)$
given by \eqref{eq:phases}.
\end{thm}


\begin{rem}
In both cases we need a certain amount of \emph{symmetric exchanges}
in the bulk
that contribute to the entropy production estimates (see Section \ref{sec:dir}).
In the Liggett case
the condition $a>1/2$ provides enough symmetric exchanges.
In the general reversible case, we accelerate the symmetric exchanges
with $\sigma_N$, and as long as condition \eqref{eq:assp-a-ge} is satisfied
there is enough entropy production.
\end{rem}

\begin{rem}
For the asymmetric current $J_{i,i+1}=\bar p\eta_i(1-\eta_{i+1})$, the quasi-static hydrodynamic limit holds also on $\Theta$: for any $\vf\in\mathcal C([0,1]\times\R_+)$, 
\begin{equation}
\label{eq:limcur}
\begin{aligned}
\lim_{N\to\infty} \int_s^{s'} 
\frac1N\sum_{i=1}^{N-1} \vf\left(\frac iN,t\right)J_{i,i+1}(t)dt 
=\:\int_s^{s'}\!\!\!\int_0^1 \vf(x,t)J(t)dx \,dt
\end{aligned}
\end{equation}
holds in probability, where $J(t)$ is given by \eqref{eq:variational}.
\end{rem}

\section{Proof of the main theorems}
\label{sec:proof}

In this section we prove Theorem \ref{thm:liggett} and \ref{thm:reversible}
by several main steps that will be proven in the following sections.
In the following we fix some $T>0$ and denote
$\Sigma_T = [0,1]\times[0,T]$. 

For $k \le N$, define the left-sided uniform block average 
\begin{align}
\label{eq:block}
\bar\eta_{i,k} := \frac1k\sum_{j=0}^{k-1} \eta_{i-j}, \quad 
\forall\,i=k,k+1,\ldots,N.
\end{align}
We choose some \emph{mesoscopic scale} $K=K(N)$ such that $K\to\infty$ and $K=o(N)$ as $N\to\infty$.
Modify the empirical process $\zeta_N$ in \eqref{eq:empirical} as 
\begin{align}
\label{eq:empirical0}
\rho_N(x,t)=\rho_{N,K(N)}(x,t) := \sum_{i=K}^N \chi_{i,N}(x) \bar\eta_{i,K}(t), \quad (x,t) \in [0,1]\times\R_+.
\end{align}

\subsection{Young measures}

By a Young measure $\nu$ we mean a family $\{\nu_{x,t}; (x,t)\in\Sigma_T\}$ 
of probability measures on $[0,1]$ such that the mapping 
\begin{align*}
(x,t) \mapsto \int_0^1 f(x,t,y)\nu_{x,t}(dy) 
\end{align*}
is measurable for all $f \in \mathcal C(\Sigma_T\times[0,1])$. 
Denote by $\mathcal Y$ the set of all Young measures, 
endowed with the vague topology: a sequence $\{\nu^n\}$ converges to some $\nu\in\mathcal Y$ as $n\to\infty$, 
if for all $g \in \mathcal C([0,1])$ and $\varphi\in\mathcal C(\Sigma_T)$,
\begin{align}
\lim_{n\to\infty} \iint_{\Sigma_T} \varphi(x,t)dx\,dt\int_0^1 g\,d\nu^n_{x,t}
= \iint_{\Sigma_T} \varphi(x,t)dx\,dt\int_0^1 g\,d\nu_{x,t}. 
\end{align}
Under this topology, $\mathcal Y$ is metrizable, separable and compact.

As $\rho_N\in[0,1]$ is uniformly bounded, the corresponding Dirac type Young measures 
\begin{align}
\nu^N := \big\{\nu_{x,t}^N=\delta_{\rho_N(x,t)}; (x,t)\in\Sigma_T\big\} \in \mathcal Y.
\end{align} 
Denote by $\mathfrak Q_N$ the distribution of $\nu^N$ on $\mathcal Y$. 

Since $\mathcal Y$ is compact, the sequence $\{\mathfrak Q_N, N\ge1\}$ is tight, 
thus we can extract a subsequence $\{N_h,h\ge1\}$ such that $\mathfrak Q_{N_h}$ converges weakly to some measure $\mathfrak Q$ on $\mathcal Y$,
namely for all $g \in \mathcal C([0,1])$ and $\varphi\in\mathcal C(\Sigma_T)$,
\begin{equation}\label{eq:33}
\begin{aligned}
\lim_{h\to\infty} \E_{N_h} 
\left[ \iint_{\Sigma_T} \varphi(x,t)g\big(\rho_{N_h}(x,t)\big)dx\,dt \right] \\
= \lim_{h\to\infty} E^{\mathfrak Q_{N_h}} 
\left[ \iint_{\Sigma_T} \varphi(x,t)dx\,dt\int g\,d\nu_{x,t} \right] \\
= E^{\mathfrak Q} \left[ \iint_{\Sigma_T} \varphi(x,t)dx\,dt\int g\,d\nu_{x,t} \right].
\end{aligned}
\end{equation}
Hereafter, we fix such a subsequence and relabel $\mathfrak Q_{N_h}$ by $\mathfrak Q_N$.

\subsection{Proof of the quasi-static limit}

In both Liggett and reversible cases, we replace the function $f(\tau_i\eta)$ in \eqref{eq:limit} with the mesoscopic canonical average $\langle f \rangle(\bar\eta_{i,K})$ by some local equilibrium argument.
Then, we prove that $\mathfrak Q$, the limit distribution of the random Young measures associated with $\bar\eta_{i,K}$, is concentrated on the quasi-static profile $\rho=\rho(t)$ given by \eqref{eq:phases}.
The proof is divided into several steps.

\subsubsection{Local equilibrium}

The next lemma allows us to replace $\eta_i$ with its mesoscopic
block average defined by \eqref{eq:block}.
It holds for both Liggett and reversible cases.

\begin{lem}
\label{lem:local-equi}
If $1 \ll K \ll \sqrt N$, then for any $\ell$ and local function
$f=f(\eta_1,\ldots,\eta_\ell)$,
\begin{align}
\label{eq:local-equi}
\lim_{N\to\infty} \E_N \left[ \int_0^T \frac1N \left| \sum_{i=0}^{N-\ell} \vf \left( \frac iN,t \right) f(\tau_i\eta) - \sum_{i=K}^N \vf \left( \frac iN,t \right) \langle f \rangle\big(\bar\eta_{i,K}\big) \right| dt \right] = 0.
\end{align}
\end{lem}

Recall the empirical process $\rho_{N,K}$ associated with $\bar\eta_{i,K}$ and the distribution $\mathfrak Q$ satisfying \eqref{eq:33}.
With Lemma \ref{lem:local-equi}, it suffices to prove that
\begin{align}
\label{eq:limit0}
\mathfrak Q\big\{\nu_{x,t}=\delta_{\rho(t)}, \ \Sigma_T-\text{a.e.}\big\}=1,
\end{align}
for some $K$ such that $1 \ll K \ll \sqrt N$,

\subsubsection{Balanced dynamics}
Recall that the dynamics is called \emph{balanced}
when $\rho_-(t)=\rho_+(t)$ for $t\in[0,T]$.

\begin{prop}
\label{prop:hl-balan}
Suppose $\rho_-(t)=\rho_+(t)=\rho(t)$ for $t\in[0,T]$.
Then \eqref{eq:limit0} holds if
\begin{align}
  &\text{for Liggett case} \quad 1 \ll K \ll \min\{N^a,N\};\\
  &\text{for reversible case} \quad 1 \ll K \ll \min\{N^a\sigma_N,\widetilde\sigma_N\sigma_N,N\}.
\end{align}
\end{prop}

Hereafter, we fix some $K$ such that $1 \ll K \ll \sqrt N$, then conditions in both Lemma \ref{lem:local-equi} and Proposition \ref{prop:hl-balan} are fulfilled.

\subsubsection{Macroscopic current}

The macroscopic current is defined by 
\begin{align}
\label{eq:quasi-current0}
\mathcal J(T) := E^{\mathfrak Q} \left[ \iint_{\Sigma_T} dx\,dt\int Jd\nu_{x,t} \right], \quad J(\rho)=\bar p\rho(1-\rho).
\end{align}
Lemma \ref{lem:local-equi} with $f=\bar p\eta_1(1-\eta_2)$, $\varphi\equiv1$
and formula \eqref{eq:33} yield that
\begin{equation}
\label{eq:quasi-current1}
\begin{aligned}
\mathcal J(T) &= \lim_{N\to\infty} \E_N \left[ \int_0^T \frac1N\sum_{i=K}^N J\big(\bar\eta_{i,K}\big)dt \right]\\
&= \lim_{N\to\infty} \E_N \left[ \int_0^T \frac1N\sum_{i=1}^{N-1} J_{i,i+1}(t)dt \right], \quad J_{i,i+1}=\bar p\eta_i(1-\eta_{i+1}).
\end{aligned}
\end{equation}
In view of Proposition \ref{prop:hl-balan}, the current for balanced dynamics reads
\begin{align}
\mathcal J(T)=\int_0^T J\big(\rho(t)\big)dt, \quad \text{when}\ \rho_\pm(t)=\rho(t).
\end{align}

\subsubsection{Coupling}

The next two lemmas are proved by coupling the unbalanced dynamics with the balanced one with boundary rate between $\rho_-(t)$ and $\rho_+(t)$.

\begin{lem}\label{lem:nu}
If $\rho_-(t)\le\rho_+(t)$ for all $t\in[0,T]$, then $\mathfrak Q$-almost surely
\begin{align}
\label{eq:nu}
\nu_{x,t}\big\{\rho\in[\rho_-(t),\rho_+(t)]\big\} = 1,
\quad (x,t)-\text{a.e. in }\Sigma_T. 
\end{align}
If $\rho_-(t)\ge\rho_+(t)$ on $[0,T]$, the above holds with $[\rho_+(t), \rho_-(t)]$. 
\end{lem}

\begin{lem}\label{lem:j}
If $\rho_-(t) \ge \rho_+(t)$ for all $t\in[0,T]$, 
\begin{align}\label{eq:j1}
\mathcal J(T) \ge \int_0^T \sup\big\{J(\rho); \rho_+(t) \le \rho \le \rho_-(t)\big\}dt. 
\end{align}
If $\rho_-(t) \le \rho_+(t)$ for all $t\in[0,T]$, 
\begin{align}\label{eq:j2}
\mathcal J(T) \le \int_0^T \inf\big\{J(\rho); \rho_-(t) \le \rho \le \rho_+(t)\big\}dt. 
\end{align}
\end{lem}


By using \eqref{eq:quasi-current0} and Lemma \ref{lem:nu} and \ref{lem:j}
we prove the following proposition:

\begin{prop}
\label{prop:j}
Suppose that $\rho_-\ge\rho_+$, a.e. in $[0,T]$ or $\rho_-\le\rho_+$, a.e. in $[0,T]$, then the equalities hold in \eqref{eq:j1} and \eqref{eq:j2}.
\end{prop}

From Proposition \ref{prop:j} it follows \eqref{eq:limcur}. 

\subsubsection{Quasi-static hydrodynamics}

We now conclude the proof of the quasi-static limits from Lemma
\ref{lem:local-equi}, \ref{lem:nu} and Proposition \ref{prop:j}.

\begin{proof}[Proofs of Theorem \ref{thm:liggett} and \ref{thm:reversible}]
Recall that $\mathfrak Q$ is the limit distribution, as $N\to\infty$, of the Young measures associated with $\rho_N$ in \eqref{eq:empirical0}.
We first prove that if $\rho_- - \rho_+$ keeps its sign in $[0,T]$, 
\begin{align}
\label{eq:concentration}
\mathfrak Q\big\{\nu_{x,t} = \delta_{\rho(t)},\ (x,t)-\text{a.e. in }\Sigma_T\big\} = 1. 
\end{align}
with $\rho(t)$ the entropy solution of \eqref{eq:qscl}
characterized by \eqref{eq:phases}.
Indeed, if $\rho_-(t) \ge \rho_+(t)$, then is a
unique $\rho(t)$ solution of
\begin{align}
J\big(\rho(t)\big) =
\sup \big\{J(\rho), \rho_+(t) \le \rho \le \rho_-(t)\big\}. 
\end{align}
By Proposition \ref{prop:j}, $\mathcal J(T)=\int_0^T J(\rho(t))dt$. 
Together with \eqref{eq:quasi-current0} and \eqref{eq:nu}, 
\begin{equation}
E^{\mathfrak Q} \left[ \iint_{\Sigma_T} dx\,dt\int_{\rho_{+}(t)}^{\rho_-(t)}
\big[J(\rho)-J(\rho(t))\big]\nu_{x,t}(d\rho) \right] = 0.
\end{equation}
Since $J(\rho)-J(\rho(t))\le0$ for $\rho\in[\rho_+(t),\rho_-(t)]$, 
the Young measure can only be concentrated in its zero set, so \eqref{eq:concentration} holds.
The argument is similar for the case $\rho_-(t) \le \rho_+(t)$ and $(\rho_-,\rho_+)(t) \notin \Theta$. 

For $f=f(\eta_1,\ldots,\eta_\ell)$ and $\varphi\in\mathcal C(\Sigma_T)$,
using Lemma \ref{lem:local-equi} and \eqref{eq:concentration}, 
\begin{align}
\lim_{N\to\infty} \E_N \left[ \int_0^T \frac1N\sum_{i=0}^{N-\ell} \vf \left( \frac iN,t \right) f(\tau_i\eta)dt \right]
= \iint_{\Sigma_T} \varphi(x,t)\langle f \rangle\big(\rho(t)\big)dx\,dt,
\end{align}
along the chosen subsequence.
The uniqueness of the entropy solution implies the uniqueness of the limit point
$\mathfrak Q$ and thus the convergence in probability.

For general $[s,s']$ where $\rho_--\rho_+$ does not change sign,
denote by $\mu_{N,t}$ the distribution of $\eta(t)$ on $\Om_N$.
Consider the process generated by $N^{1+a}L_{N,t}$, $0 \le t \le s'-s$ and initial distribution $\mu_{N,s}$. 
As the arguments above are valid for any fixed time interval and initial distribution, the result holds almost surely in $[s,s']$. 
\end{proof}

The remaining part of this article is organized as follows.
We first present in Section \ref{sec:currents} and \ref{sec:dir} some estimates on the microscopic currents and the Dirichlet forms.
These bounds are used repeatedly throughout the article.
Then in Section \ref{sec:local-equi} we prove Lemma \ref{lem:local-equi} as a consequence of a general one-block estimate.

In Section \ref{sec:balance}, we treat the balanced dynamics and prove Proposition \ref{prop:hl-balan}.
The proof exploits the boundary entropy--entropy flux pair
defined in Definition \ref{bLef}.
In fact in the balanced case, we expect equality in \eqref{eq:otto},
and we see in Section \ref{sec:bentropy} 
that the time average of the boundary entropy flux 
on the microscopic scale is negligible, 
see Proposition \ref{prop:bd-ent-li} and \ref{prop:bd-ent-ge}.
This is the hard part of the proof.
Finally in Section \ref{sec:coupling} we deal with the unbalanced dynamics.
We prove there Lemma \ref{lem:nu} and \ref{lem:j} by constructing a coupling between the unbalanced and balanced dynamics.

\section{Microscopic currents}
\label{sec:currents}

The \emph{microscopic currents} $j_{i,i+1}$ associated to the generator $L_{N,t}$ are defined by the conservation law $L_{N,t}[\eta_i] = j_{i-1,i}-j_{i,i+1}$ and they are equal to 
\begin{equation}
\label{eq:current-li}
j_{i,i+1} = 
\begin{cases}
p\rho_-(t) - \big[p\rho_-(t)+(1-p)(1-\rho_-(t))\big]\eta_1, &i=0, \\
J_{i,i+1} + \frac{1-\bar p}2(\eta_i - \eta_{i+1}), \quad J_{i,i+1}=\bar p\eta_i(1-\eta_{i+1}), &1 \le i \le N-1, \\
\big[p(1-\rho_+(t))+(1-p)\rho_+(t)\big]\eta_N - (1-p)\rho_+(t), &i=N. 
\end{cases}
\end{equation}
in the case of Liggett boundary rates and 
\begin{equation}
\label{eq:current-ge}
j_{i,i+1} = 
\begin{cases}
\widetilde\sigma_N\la_-(t)(\rho_-(t)-\eta_1), &i=0, \\
J_{i,i+1} + \frac{\sigma_N-\bar p}2(\eta_i - \eta_{i+1}), &1 \le i \le N-1, \\
\widetilde\sigma_N\la_+(t)(\eta_N-\rho_+(t)), &i=N. 
\end{cases}
\end{equation}
in the case of reversible boundary rates. 

Follow the argument as in \cite[Section 2]{DO}, 
for $i=1$, ... $N-1$ define the counting processes 
associated to the process $\{\eta(t)\}_{t\ge 0}$ generated by $N^{1+a}L_{N,t}$ by 
\begin{equation}\label{eq:counting1}
\begin{aligned}
&h_+(i,t) := \text{number of jumps } i \to i+1 
\text{ in } [0,t]. \\
&h_-(i,t) := \text{number of jumps } i+1 \to i 
\text{ in } [0,t]. \\
&h(i,t) := h_+(i,t) - h_-(i,t). 
\end{aligned}
\end{equation}
These definitions extend to the boundaries $i=0$ and $i=N$ as 
\begin{equation}\label{eq:counting2}
\begin{aligned}
&h_+(0,t) := \text{number of particles created at } 1 
\text{ in } [0,t], \\
&h_-(0,t) := \text{number of particles annihilated at } 1 
\text{ in } [0,t], \\
&h_+(N,t) := \text{number of particles annihilated at } N 
\text{ in } [0,t], \\
&h_-(N,t) := \text{number of particles created at } N 
\text{ in } [0,t]. 
\end{aligned}
\end{equation}
The conservation law is microscopically given by 
\begin{equation}\label{eq:claw0}
\eta_i(t) - \eta_i(0) = h(i-1,t) - h(i,t), \quad \forall\,i= 1,\dots,N. 
\end{equation}
Furthermore, for $i=0,\dots,N$ there is a martingale $M_i(t)$ such that 
\begin{align}
\label{eq:5-0}
h(i,t) = N^{1+a} \int_0^t j_{i,i+1}(s)ds +M_i(t). 
\end{align}
As $|\eta_i(t)| \le 1$, \eqref{eq:claw0} yields that $|h(i,t) - h(i',t)| \le |i-i'|$. 
Therefore, 
\begin{align}
\label{eq:5-1}
\E_N \left[\int_0^t \big(j_{i,i+1}(s) - j_{i',i'+1}(s)\big)ds\right] = O(N^{-a}), \quad \forall\,i,i'=0,1,\ldots,N. 
\end{align}

Now fix $T>0$ and recall the macroscopic current defined
in \eqref{eq:quasi-current0}.
By \eqref{eq:quasi-current1}, \eqref{eq:current-li} and \eqref{eq:current-ge},
\begin{align}
\label{eq:5-2}
\left| \E_N \left[ \int_0^T \frac1N\sum_{i=1}^{N-1} j_{i,i+1}(t)dt \right]
- \mathcal J(T) \right| 
\mathop{\longrightarrow}_{N\to\infty} 0.
\end{align}
Hence, one concludes from \eqref{eq:5-1} and \eqref{eq:5-2} that 
\begin{align}
\label{eq:quasi-current}
\mathcal J(T) = \lim_{N\to\infty} \E_N \left[ \int_0^T j_{i,i+1}(t)dt \right], \quad \forall i=0,1,\ldots,N.
\end{align}

In particular for reversible case, since $i=0$, $N$ are included in \eqref{eq:5-1} and \eqref{eq:quasi-current},
\begin{equation}\label{eq:current-bd}
\begin{aligned}
&\widetilde\sigma_N\E_N \left[ \int_0^T \la_-(t)\big(\eta_1(t) - \rho_-(t)\big)dt \right] = -\mathcal J(T) + O(N^{-a}), \\
&\widetilde\sigma_N\E_N \left[ \int_0^T \la_+(t)\big(\eta_N(t) - \rho_+(t)\big)dt \right] = \mathcal J(T) + O(N^{-a}). 
\end{aligned}
\end{equation}
Since $\inf_{[0,T]} \la_\pm > 0$, we obtain the boundary estimates 
\begin{align*}
  \left| \E_N \left[ \int_0^T \big(\eta_1(t) - \rho_-(t)\big)dt \right] \right| \le \frac C{\widetilde\sigma_N}, \quad 
  \left| \E_N \left[ \int_0^T \big(\eta_N(t) - \rho_+(t)\big)dt \right] \right| \le \frac C{\widetilde\sigma_N}.
\end{align*}
Since $T$ is arbitrary, this means the acceleration of the boundary rates
($\widetilde\sigma_N\to\infty$) forces the boundary conditions
$\rho_\pm (t)$ in this microscopic sense.

\section{Microscopic entropy production: bounds on
the Dirichlet forms}
\label{sec:dir}

For any $\rho\in[0,1]$, let $\nu_\rho$ be the product Bernoulli measure on $\Om_N=\{0,1\}^N$ with rate $\rho$. 
In particular, for the time dependent parameters $\rho_\pm=\rho_\pm(t)$ denote 
\begin{equation}\label{eq:bbern}
\nu_{\pm,t}(\eta) := \nu_{\rho_\pm(t)}(\eta) =
\prod_{i=1}^N\,\rho_\pm(t)^{\eta_i}\big[1-\rho_\pm(t)\big]^{1-\eta_i}, \quad \forall\,\eta \in \Om_N. 
\end{equation}
Recall that $\eta(t)$ is the process generated by $N^{1+a}L_{N,t}$ 
and denote by $\mu_{N,t}$ the distribution of $\eta(t)$ in $\Omega_N$. 
For $N \ge 2$, define the Dirichlet form 
\begin{align}
\label{eq:dir-exc}
\D_{\exc,N}(t) := \frac12\sum_{\eta\in\Om_N} \sum_{i=1}^{N-1} 
\left(\sqrt{\mu_{N,t}(\eta^{i,i+1})} - \sqrt{\mu_{N,t}(\eta)}\right)^2. 
\end{align}
Let $f^\pm_{N,t}$ be the density of $\mu_{N,t}$ with respect to $\nu_{\pm,t}$ 
and define the boundary Dirichlet forms as 
\begin{equation}
\begin{aligned}
&\D_{-,N}(t) := \frac12\sum_\eta \rho_-^{1-\eta_1}(1-\rho_-)^{\eta_1}
\left(\sqrt{f_{N,t}^-(\eta^1)} - \sqrt{f_{N,t}^-(\eta)}\right)^2\nu_{-,t}(\eta), \\
&\D_{+,N}(t) := \frac12\sum_\eta \rho_+^{1-\eta_N}(1-\rho_+)^{\eta_N}
\left(\sqrt{f_{N,t}^+(\eta^N)} - \sqrt{f_{N,t}^+(\eta)}\right)^2\nu_{+,t}(\eta). 
\end{aligned}
\end{equation}
In this section we establish some useful bounds for these Dirichlet forms. 
We start from the Liggett case and prove the next result. 

\begin{prop}[Liggett boundary]
\label{prop:dir1}
For all $t<t'$, there exists $C$ such that 
\begin{align}
\label{eq:dir1}
\int_t^{t'} \big[\D_{-,N}(s)+\D_{\exc,N}(s)+\D_{+,N}(s)\big]ds \le C 
\end{align}
for all $N$. Moreover if $\rho_-(s)=\rho_+(s)$ for all $s\in[t,t']$, then 
\begin{align}
\label{eq:dir1b}
\int_t^{t'} \big[\D_{-,N}(s)+\D_{\exc,N}(s)+\D_{+,N}(s)\big]ds \le \frac C{N^a}. 
\end{align}
\end{prop}

\begin{proof}
In this proof,
any function $f$ on $\Om_N$ is viewed as a local function on
$\{0,1\}^{\mathbb Z}$ in the standard way.

Given any probability measure $\mu$ on $\Om_N$, for a given $t\ge 0$
we extend it as a measure $\bar\mu$ on
$\{0,1\}^{\mathbb Z}\sim\Om_N \times \prod_{i\notin \{1,..,N\}}\{0,1\}$
where outside $\{1,..,N\}$ is a product measure with 
$\bar\mu\{\eta_i=1\} = \rho_-(t)$ for $i\le 0$ and 
$\bar\mu\{\eta_i=1\} = \rho_+(t)$ for $i>N$. 

Recall that $L_{N,t}$ is the generator with Liggett boundary rates in \eqref{eq:liggett}. 
Also define for all local function $f$ on $\{0,1\}^\mathbb Z$ 
\begin{align*}
L_\exc^\mathbb Zf := \sum_{i \in\mathbb Z} c_{i,i+1} \left[ f\left(\eta^{i,i+1}\right)-f(\eta) \right], \quad c_{i,j} = p\eta_i(1-\eta_j) + (1-p)\eta_j(1-\eta_i). 
\end{align*}
$L_\exc^\mathbb Z$ generates the asymmetric simple exclusion
process on $\{0,1\}^\mathbb Z$. 
As obtained in \cite[p. 244]{Liggett75} in the proof of Theorem 2.4,
for a function $g=g(\eta_1,\ldots,\eta_N)$, 
\begin{equation}
\begin{aligned}\label{eq:ligg}
&L_{N,t}g(\eta) - L_\exc^\mathbb Z g(\eta) \\
= &\,\big[p(\rho_--\eta_0)(1-\eta_1)+(1-p)(\eta_0-\rho_-)\eta_1\big] \big[g(\eta^1)-g(\eta)\big] \\
&+ \big[(1-p)(\rho_+-\eta_{N+1})(1-\eta_N)+p(\eta_{N+1}-\rho_+)\eta_N\big] \big[g(\eta^N)-g(\eta)\big]. 
\end{aligned}
\end{equation}
Observe that the integral of the right-hand side is $0$ with respect to a measure
$\bar\mu$ on $\{0,1\}^{\mathbb Z}$ obtained by extending any measure
$\mu$ on $\Om_N$ as above. 

Recall the Bernoulli measure $\nu_{\pm,t}$ defined in \eqref{eq:bbern}
and the corresponding density function $f_{N,t}^\pm$: $\mu_{N,t} = f_{N,t}^\pm \nu_{\pm,t}$. 
Consider the relative entropy $H_{\pm,N}(t)$ given by 
\begin{align*}
H_{\pm,N}(t) := \sum_{\eta\in\Om_N} f_{N,t}^\pm(\eta)\log f_{N,t}^\pm(\eta)\nu_{\pm,t}(\eta)
= \sum_{\eta\in\Om_N}\log f_{N,t}^\pm(\eta) \mu_{N,t} (\eta). 
\end{align*}
Using Kolmogorov equation and \eqref{eq:ligg}, we obtain for the entropy production that 
\begin{equation}
\label{eq:h1}
\begin{aligned}
\frac d{dt}H_{-,N}(t) &= N^{1+a}\sum_{\eta\in\Om_N} L_{N,t} [\log f_{N,t}^-]\mu_{N,t} 
- \sum_{\eta\in\Om_N} f_{N,t}^-\frac d{dt}\nu_{-,t} \\
&= N^{1+a}\sum_{\eta\in\{0,1\}^\mathbb Z} L_\exc^{\mathbb Z} [\log f_{N,t}^-]\bar\mu_{N,t}
- \sum_{\eta\in\Om_N} f_{N,t}^-\frac d{dt}\nu_{-,t}\\
&= N^{1+a} \sum_{\eta\in\{0,1\}^\mathbb Z} f_{N,t}^-L_\exc^\mathbb Z \big[\log f_{N,t}^-\big]
\bar\nu_{-,t}
- \sum_{\eta\in\Om_N} f_{N,t}^-\frac d{dt}\nu_{-,t}, 
\end{aligned}
\end{equation}
where $ \bar\nu_{-,t} = \prod_{i\le N} \nu_{\rho_{-}(t)}(\eta_i)
\prod_{i>N}\nu_{\rho_+(t)}(\eta_i)$.
Observe that the last term in \eqref{eq:h1} is bounded by $CN$.

Exploiting the inequality $x(\log y-\log x)\le2\sqrt x(\sqrt y-\sqrt x)$ for all $x$, $y>0$,
and denoting that $\bar p = 2p-1$, $g_{N,t}^-=(f_{N,t}^-)^{1/2}$, 
\begin{equation}
\label{eq:h2}
\begin{aligned}
&\sum_{\eta\in\{0,1\}^\mathbb Z} f_{N,t}^-L_\exc^\mathbb Z \big[\log f_{N,t}^-\big]
\bar\nu_{-,t} \le
2\sum_{\eta\in\{0,1\}^\mathbb Z} g_{N,t}^- L_\exc^\mathbb Z \big[g_{N,t}^-\big]
\bar\nu_{-,t} \\
=&\sum_{\eta\in\{0,1\}^\mathbb Z} \sum_{i\in\mathbb Z} \big[1+\bar p(\eta_i-\eta_{i+1})\big] \left[g_{N,t}^-(\eta^{i,i+1})
-g_{N,t}^-\right]g_{N,t}^-\bar\nu_{-,t} \\
=&\:\mathcal I_1+\mathcal I_2+\mathcal I_3+\mathcal I_4, 
\end{aligned}
\end{equation}
where the right-hand side reads 
\begin{align*}
\mathcal I_1 = -\frac12\sum_{\eta\in\{0,1\}^\mathbb Z}
\sum_{i\in\mathbb Z} \left[g_{N,t}^-(\eta^{i,i+1})
-g_{N,t}^-(\eta)\right]^2\bar\nu_{-,t}(\eta), 
\end{align*}
\begin{align*}
\mathcal I_2 &= -\frac12\sum_{\eta\in\{0,1\}^\mathbb Z} \sum_{i\in\mathbb Z} (g_{N,t}^-)^2(\eta)\big[\bar\nu_{-,t}(\eta)-\bar\nu_{-,t}(\eta^{i,i+1})\big] \\
&= \frac12\sum_{\eta\in\{0,1\}^\mathbb Z} f_{N,t}^-(\eta)\big[\bar\nu_{-,t}(\eta^{N,N+1})-\bar\nu_{-,t}(\eta)\big], 
\end{align*}
\begin{align*}
\mathcal I_3 &= \bar p\sum_{\eta\in\{0,1\}^\mathbb Z} \sum_{i\in\mathbb Z} (\eta_i-\eta_{i+1})g_{N,t}^-(\eta^{i,i+1})g_{N,t}^-\bar\nu_{-,t} \\
&= \frac{\bar p}2\sum_{\eta\in\{0,1\}^\mathbb Z} (\eta_{N+1}-\eta_N)g_{N,t}^-(\eta^{N,N+1})g_{N,t}^-\big[\bar\nu_{-,t}(\eta^{N,N+1})-\bar\nu_{-,t}\big], 
\end{align*}
\begin{align*}
\mathcal I_4 &= \bar p\sum_{\eta\in\{0,1\}^\mathbb Z} \sum_{i\in\mathbb Z} (\eta_i-\eta_{i+1})\left[-(g_{N,t}^-)^2(\eta)\right]\bar\nu_{-,t}(\eta) = \bar p\big[\rho_+(t)-\rho_-(t)\big]. 
\end{align*}
Notice that $\mathcal I_2$, $\mathcal I_3$, $\mathcal I_4$ are uniformly bounded in $N$ and they vanish when $\rho_-=\rho_+$. 
On the other hand, since $f_{N,t}^-$ depends only on $\{\eta_1,\ldots,\eta_N\}$, 
\begin{align*}
\mathcal I_1 = -\D_{\exc,N}(t) - \mathcal I_{1,l} - \mathcal I_{1,r}, 
\end{align*}
where $\mathcal I_{1,l}$ and $\mathcal I_{1,r}$ are computed respectively as 
\begin{align*}
\mathcal I_{1,l} &= \frac12\sum_{\eta\in\Om_N} 
\left[g_{N,t}^-(\eta^1) -g_{N,t}^-\right]^2\nu_{-,t}(\eta)
\sum_{\eta_0} \big(\eta_0(1 -\eta_1) + \eta_1(1-\eta_0)\big)\nu_{\rho_-(t)} (\eta_0) \\
&= \frac12\sum_{\eta\in\Om_N} \big(\rho_-(t) (1 -\eta_1) + \eta_1(1-\rho_-(t))\big)
\left[g_{N,t}^-(\eta^1) -g_{N,t}^-\right]^2\nu_{-,t} = \D_{-,N}(t), 
\end{align*}
\begin{align*}
\mathcal I_{1,r} = \frac12\sum_{\eta\in\Om_N} 
&\left[g_{N,t}^-(\eta^N)-g_{N,t}^-\right]^2\nu_{-,t}(\eta) \\
&\sum_{\eta_{N+1}} \big(\eta_N(1-\eta_{N+1}) + \eta_{N+1}(1-\eta_N)\big)
\nu_{\rho_+(t)}(\eta_{N+1}) \ge 0. 
\end{align*}
Hence, from \eqref{eq:h2} we obtain a constant $C>0$ such that 
\begin{align*}
\sum_{\eta\in\{0,1\}^\mathbb Z} f_{N,t}^-L_\exc^\mathbb Z \big[\log f_{N,t}^-\big]\nu_{-,t} \le
- \D_{-,N}(t) - \D_{\exc,N}(t) + C. 
\end{align*}
Plugging this into \eqref{eq:h1} and integrating in time, 
\begin{align*}
\int_0^t \big[\D_{-,N}(s) + \D_{\exc,N}(s)\big]ds \le C. 
\end{align*}
The proof of \eqref{eq:dir1} is completed by repeating the argument with $H_{N,t}^+$. 

Now suppose that $\rho_-(t)=\rho_+(t)$, then $\nu_{-,t}=\nu_{+,t}$ and we have $\mathcal I_{1,r} = \D_{+,N}(t)$. 
Therefore, \eqref{eq:h2} yields that 
\begin{align*}
\sum_{\eta\in\mathbb Z} f_{N,t}L_\exc^\mathbb Z \big[\log f_{N,t}\big]\nu_t \le -\D_{-,N}(t)-\D_{\exc,N}(t)-\D_{+,N}(t). 
\end{align*}
Since $H_N(0)=O(N)$, \eqref{eq:dir1b} follows similarly. 
\end{proof}

For reversible boundaries, the following estimate holds instead. 

\begin{prop}[Reversible boundary]
\label{prop:dir2}
For all $t<t'$, there exists $C$ such that 
\begin{align}
\label{eq:dir2}
\int_t^{t'} \big[\widetilde\sigma_N\la_-(s)\D_{-,N}(s) + \sigma_N\D_{\exc,N}(s) + \widetilde\sigma_N\la_+(s)\D_{+,N}(s)\big]ds \le C 
\end{align}
for all $N$. Moreover if $\rho_-(s)=\rho_+(s)$ for all $s\in[t,t']$, then 
\begin{equation}
\label{eq:dir2b}
\begin{aligned}
&\int_t^{t'} \big[\widetilde\sigma_N\la_-(s)\D_{-,N}(s) + \sigma_N\D_{\exc,N}(s) + \widetilde\sigma_N\la_+(s)\D_{+,N}(s)\big]ds \\
\le\:&\frac{\bar p}{\widetilde\sigma_N}\int_0^t \left[ \frac{\mathcal J(s)}{\la_-(s)} + \frac{\mathcal J(s)}{\la_+(s)} \right] ds + \frac C{N^a} \le C' \left( \frac1{\widetilde\sigma_N} + \frac1{N^a} \right). 
\end{aligned}
\end{equation}
\end{prop}

\begin{proof}
Similarly to \eqref{eq:h1}, we obtain for the entropy production that 
\begin{equation}
\label{eq:h3}
\frac d{dt}H_{-,N}(t) \le N^{1+a}\sum_{\eta\in\Om_N} f_{N,t}^-L_{N,t} [\log f_{N,t}^-]\nu_{-,t} + CN. 
\end{equation}
Applying the argument used in \eqref{eq:h2}, 
\begin{align*}
&\sum_{\eta\in\Om_N} f_{N,t}^- \big(L_{N,t} - \widetilde\sigma_N\la_+(t)L_{+,t}\big) [\log f_{N,t}^-]\nu_{-,t} \\
\le\,&-\widetilde\sigma_N\la_-(t)\D_{-,N}(t) - \sigma_N\D_{\exc,N}(t) + \bar p\sum_{\eta\in\Om_N} (\eta_N-\eta_1)f_{N,t}^-\nu_{-,t}. 
\end{align*}
In view of \eqref{eq:current-bd}, the last term can be bounded as 
\begin{align*}
\bar p\sum_\eta (\eta_N-\eta_1)f_{N,t}^-\nu_{-,t} 
= \bar p \left[ \rho_+(t) - \rho_-(t) + \frac{\mathcal J(t)}{\widetilde\sigma_N\la_-(t)} + \frac{\mathcal J(t)}{\widetilde\sigma_N\la_+(t)} \right] + \frac C{N^a}. 
\end{align*}
For $L_{+,t}$, since $f_{N,t}^-\nu_{-,t}=f_{N,t}^+\nu_{+,t}$, 
\begin{align*}
\sum_\eta f_{N,t}^- L_{+,t} [\log f_{N,t}^-]\nu_{-,t} \le - \D_{+,N}(t) + \sum_\eta f_{N,t}^- L_{+,t} \left[ \log \left( \frac{\nu_{+,t}}{\nu_{-,t}} \right) \right] \nu_{-,t}. 
\end{align*}
Noting that $\rho_\pm \in (0,1)$, standard manipulation shows that 
\begin{align*}
L_{+,t} \left[ \log \left( \frac{\nu_{+,t}}{\nu_{-,t}} \right) \right] = -\left\{ \log \left[ \frac{\rho_+(t)}{1-\rho_+(t)} \right] - \log \left[ \frac{\rho_-(t)}{1-\rho_-(t)} \right] \right\} (\eta_N-\rho_+(t)). 
\end{align*}
Let $F(\rho) = \rho\log\rho + (1-\rho)\log(1-\rho)$, so that $F'(\rho) = \log(\rho/(1-\rho))$. 
By \eqref{eq:current-bd}, 
\begin{align*}
\sum_\eta f_{N,t}^- L_{+,t} \left[ \log \left( \frac{\nu_{+,t}}{\nu_{-,t}} \right) \right] \nu_{-,t} \le -\big[F'(\rho_+(t))-F'(\rho_-(t))\big]\frac{\mathcal J(t) + O(N^{-a})}{\widetilde\sigma_N\la_+(t)}. 
\end{align*}
Putting all these estimates together, we obtain from \eqref{eq:h3} that 
\begin{align*}
\frac d{dt}H_{-,N}(t) \le &-N^{1+a}\big[2\widetilde\sigma_N\la_-(t)\D_{-,N}(t) + \sigma_N\D_{\exc,N}(t) + 2\widetilde\sigma_N\la_+(t)\D_{+,N}(t)\big] \\
&+ N^{1+a}\big[C(\rho_\pm,t) + \widetilde\sigma_N^{-1}C(\la_\pm,t)\big] + CN, 
\end{align*}
where $C(\rho_\pm,t)$ and $C(\la_\pm,t)$ are constants given by 
\begin{align*}
C(\rho_\pm,t) &= \bar p\big[\rho_+(t)-\rho_-(t)\big] - \mathcal J(t)\big[F'(\rho_+(t)) - F'(\rho_-(t))\big], \\
C(\la_\pm,t) &= \bar p\mathcal J(t)\big[\la_+^{-1}(t)+\la_-^{-1}(t)\big]. 
\end{align*}
We can then conclude \eqref{eq:dir2} by integrating in time. 
For \eqref{eq:dir2b}, it suffices to observe that $C(\rho_\pm,t)=0$ when $\rho_-(t)=\rho_+(t)$. 
\end{proof}

\begin{rem}
The condition $\rho_-(s)=\rho_+(s)$ is necessary for $C(\rho_\pm,s) = 0$. 
Indeed, in view of Proposition \ref{prop:j}, if $\rho_+(s)<\rho_-(s)$, $\mathcal J(s) = \sup \{J(\rho); \rho_+(s)\le\rho\le \rho_-(s)\}$, so 
\begin{align*}
C(\rho_\pm,s) &= \int_{\rho_+(s)}^{\rho_-(s)} \left(\frac{\mathcal J(s)}{\rho(1-\rho)} - \bar p\right)d\rho = \int_{\rho_+(s)}^{\rho_-(s)} \frac{\mathcal J(s)-J(\rho)}{\rho(1-\rho)}d\rho > 0. 
\end{align*}
Meanwhile if $\rho_+(s) \ge \rho_-(s)$, $\mathcal J(s) = \inf \{J(\rho); \rho_-(s)\le\rho\le\rho_+(s)\}$, so 
\begin{align*}
C(\rho_\pm,s) &= \int_{\rho_-(s)}^{\rho_+(s)} \left(\bar p - \frac{\mathcal J(s)}{\rho(1-\rho)}\right)d\rho = \int_{\rho_-(s)}^{\rho_+(s)} \frac{J(\rho)-\mathcal J(s)}{\rho(1-\rho)}d\rho > 0. 
\end{align*}
Hence, better bounds are available only when $\rho_-=\rho_+$. 
\end{rem}

\section{Local equilibrium}
\label{sec:local-equi}

Recall that $\mu_{N,t}$ is the distribution of the dynamics at time $t$.

We prove in this section Lemma \ref{lem:local-equi},
as a direct consequence of the following \emph{one-block estimate}.

\begin{prop}\label{prop:one-block}
There exists some constant $C$, such that for any $\ell\ge1$, function $f$ supported on $\{\eta_1,\ldots,\eta_\ell\}$, $N \ge k \ge \ell$ and $t\ge0$,
\begin{align}
\label{eq:one-block}
  \int \sum_{i=k}^N \big|f_{i-\ell,k-\ell}-\langle f \rangle(\bar\eta_{i,k})\big|^2d\mu_{N,t} \le \frac{C\ell[k^3\D_{\exc,N}(t) + N]}{k-\ell+1}, 
\end{align}
where $f_{i,k} := (k+1)^{-1}\big[f(\tau_{i-k}\eta)+f(\tau_{i-k+1}\eta)+\ldots+f(\tau_i\eta)\big]$.
\end{prop}

\begin{rem}
\label{rem:one-block0}
From the proof below, it is easy to see that the uniform average $f_{i,k}$ in can be replaced with weighted average $f_{i,k}^*=\sum_{0 \le j \le k} a_{k,j}f(\tau_{i-j}\eta)$ for any weight $\{a_{k,j}\}$ such that $\sum_j a_{k,j}=1$ and $\sum_j a_{k,j}^2 = O(k^{-1})$.
\end{rem}




\begin{proof}
For $\rho_* \in I_k := \{i/k;i=0,1,\ldots,k\}$, define
\begin{align}
\Om_{k,\rho_*} := \left\{\eta=(\eta_1,\dots,\eta_k) \in \Om_k~\bigg|~\frac1k\sum_{i=1}^k \eta_i = \rho_*\right\}.
\end{align}
Let $\nu^k(\,\cdot\,|\rho_*)$ be the uniform measure on $\Om_{k,\rho_*}$.
For $i=k$, $k+1$, ..., $N$, let
\begin{equation}
\begin{aligned}
&\bar\mu_{N,t}^{i,k}(\rho_*) := \mu_{N,t}\left\{(\eta_{i-k+1},\ldots,\eta_i)\in\Om_{k,\rho_*}\right\}, \\
&\mu_{N,t}^{i,k}(\eta_{i-k+1},\ldots,\eta_i|\rho_*) := \frac{\mu_{N,t}(\eta_{i-k+1},\ldots,\eta_i)}{\bar\mu_{N,t}^{i,k}(\rho_*)}. 
\end{aligned}
\end{equation}
Write $\mathcal F_{i,k}:=f_{i-\ell,k-\ell}-\langle f \rangle(\bar\eta_{i,k})$. 
By the relative entropy inequality, 
\begin{equation}\label{eq:rel-ent-ineq}
\begin{aligned}
\int \big|\mathcal F_{i,k}\big|^2d\mu_{N,t} =&\sum_{\rho_* \in I_k} \bar\mu_{N,t}^{i,k}(\rho_*)\int \big|\mathcal F_{i,k}\big|^2d\mu_{N,t}^{i,k}(\,\cdot\,|\rho_*) \\
\le\:&\frac1a\sum_{\rho_* \in I_k} \bar\mu_{N,t}^{i,k}(\rho_*)\bigg\{H\left(\mu_{N,t}^{i,k}(\,\cdot\,|\rho_*); \nu^k(\,\cdot\,|\rho_*)\right) \\
&+\:\log\int \exp\big\{a|\mathcal F_{i,k}|^2\big\}d\nu^k(\,\cdot\,|\rho_*)\bigg\}, \quad \forall\,a>0, 
\end{aligned}
\end{equation}
where $H$ is the relative entropy: for two measures $\mu$, $\nu$ on $\Om_{k,\rho_*}$, 
\begin{align}
H(\mu;\nu) := \int (\log\mu - \log\nu)d\mu. 
\end{align}
The logarithmic Sobolev inequality \eqref{eq:logsob} yields that there is a universal constant $C_\LS$, such that for each $i$, $k$ and $\rho_*$, 
\begin{align}\label{eq:logsob1}
H\left(\mu_{N,t}^{i,k}(\,\cdot\,|\rho_*); \nu^k(\,\cdot\,|\rho_*)\right) \le 2^{-1}C_\LS k^2\D_{N,\rho_*}^{i,k}(t), 
\end{align}
where the Dirichlet form in the right-hand side is defined as 
\begin{align}
\D_{N,\rho_*}^{i,k}(t) := \frac12\sum_{\eta\in\Om_{k,\rho_*}}
\sum_{i'=1}^{k-1} \left(\sqrt{\mu_{N,t}^{i,k}(\eta^{i',i'+1}|\rho_*)}
- \sqrt{\mu_{N,t}^{i,k}(\eta|\rho_*)}\right)^2. 
\end{align}
Plugging \eqref{eq:logsob1} into \eqref{eq:rel-ent-ineq}
and using Schwarz inequality, we obtain
\begin{equation}
\begin{aligned}
\sum_{i=k}^N \int \big|\mathcal F_{i,k}\big|^2&d\mu_{N,t} \le \frac{C_\LS k^3}{2a}\D_{\exc,N}(t) \\
&+\frac1a\sum_{i=k}^N \sum_{\rho_* \in I_k} \bar\mu_{N,t}^{i,k}(\rho_*) \log\int e^{a\mathcal F_{i,k}^2}d\nu^k(\,\cdot\,|\rho_*). 
\end{aligned}
\end{equation}
The desired estimate then follows if we can find constants $c$, $C$ such that 
\begin{align}\label{eq:exp}
\int \exp\big\{a|\mathcal F_{i,k}|^2\big\}d\nu^k(\,\cdot\,|\rho_*) \le C, \quad \forall\,a < \frac{ck}\ell.
\end{align}

We are left with the proof of \eqref{eq:exp}. 
Without loss of generality, we assume that the local function $f \in [0,1]$. 
By Hoeffding's lemma, for all $\rho \in [0,1]$, 
\begin{align}
\log\int e^{a[f-\langle f \rangle(\rho)]}d\nu_\rho \le \frac{a^2}8, \quad \forall\,a \in \R. 
\end{align}
By splitting the family $\{f(\tau_{i'}\eta),i'=i-k,\ldots,i-\ell\}$ into independent groups and applying the generalized H\"older's inequality, 
\begin{align}
\log\int \exp\left\{\frac a{k-\ell+1}\sum_{i'=i-k}^{i-\ell} \big[f(\tau_{i'}\eta) - \langle f \rangle(\rho)\big]\right\}d\nu_\rho \le \frac{\ell a^2}{8(k-\ell+1)}. 
\end{align}
Standard manipulation then shows that if $a \le \ell^{-1}(k-\ell+1)$, 
\begin{align}
\log\int \exp\left\{a\bigg|\frac1{k-\ell+1}\sum_{i'=i-k}^{i-\ell} \big[f(\tau_{i'}\eta) - \langle f \rangle(\rho)\big]\bigg|^2\right\}d\nu_\rho \le 3. 
\end{align}
In order to obtain \eqref{eq:exp} it suffices to replace $\nu_\rho$ with its conditional measure $\nu^k(\,\cdot\,|\rho_*)$. 
It follows from an elementary estimate that 
\begin{align}
\nu^k\big(\eta|_\Ga = \tilde\eta\,\big|\,\rho_*\big) \le C\nu_{\rho_*}(\eta|_\Ga = \tilde\eta) 
\end{align}
for all $\Ga \subset \{i-k+1,\ldots,i\}$ such that $|\Ga| \le 2k/3$. 
\end{proof}

With Proposition \ref{prop:one-block} and the estimates on the Dirichlet forms proved in Section \ref{sec:dir}, one easily concludes Lemma \ref{lem:local-equi}.

\begin{proof}[Proof of Lemma \ref{lem:local-equi}]
Let $g=\langle f \rangle$.
As $K\gg\ell$, Proposition \ref{prop:one-block} yields that
\begin{align}
\E_N \left[ \int_0^T \frac1N\sum_{i=K}^N \big|f_{i-\ell,K-\ell}-g\big(\bar\eta_{i,K}\big)\big|^2dt \right] \le C' \int_0^T \left[ \frac{K^2}N\D_{\exc,N}(t) + \frac1K \right] dt.
\end{align}
In view of \eqref{eq:dir1} and \eqref{eq:dir2}, for $K$ such that $1 \le K \le \sqrt N$, the expression above vanishes as $N\to\infty$.
Using Schwarz inequality,
\begin{align}
\label{eq:7-1}
\lim_{N\to\infty} \E_N \left[ \int_0^T \frac1N\sum_{i=K}^N \vf \left( \frac iN,t \right) \big[f_{i-\ell,K-\ell}-g\big(\bar\eta_{i,K}\big)\big]dt \right] = 0.
\end{align}
Meanwhile, since $K \ll N$ and $\vf$ is continuous, we can replace $f(\tau_i\eta)$ by its block average with a uniformly vanishing error for each $t\in[0,T]$:
\begin{align}
\label{eq:7-2}
\lim_{N\to\infty} \frac1N\sum_{i=K}^N \vf \left( \frac iN,t \right) \big[ f(\tau_{i-\ell}\eta) - f_{i-\ell,K-\ell} \big] = 0.
\end{align}
Lemma \ref{lem:local-equi} follows from \eqref{eq:7-1} and \eqref{eq:7-2}.
\end{proof}

\section{Balanced dynamics}
\label{sec:balance}

In this section, we focus on the \emph{balanced dynamics}: $\rho_-(t)=\rho_+(t)=\rho(t)$ for $t\in[0,T]$.
Our aim is to prove Proposition \ref{prop:hl-balan} for both Liggett case and reversible case.

Recall the uniform block averages $\bar\eta_{i,k}$ defined by \eqref{eq:block} and the mesoscopic scale $K=o(N)$.
Define the \emph{smoother weighted averages} 
\begin{align}
\label{eq:smooth-block}
\etabar_{i,K} := \frac1K\sum_{j=0}^{K-1} \bar\eta_{i+j,K} = \frac1{K^2}\sum_{|j|<K} \big(K-|j|\big)\eta_{i-j}, \quad i=K,\ldots,N-K+1.
\end{align}
The empirical process $\hat\rho_N=\hat\rho_{N,K(N)}$ associated with $\etabar_{i,K}$ reads 
\begin{align}\label{eq:empirical1}
\hat\rho_N(x,t) := \sum_{i=K+1}^{N-K} \chi_{i,N}(x) \etabar_{i,K}(t), \quad (x,t) \in [0,1]\times\R_+, 
\end{align}
with $\chi_{i,N}$ in \eqref{eq:indicator}.
Observe that $(\eta_1,\eta_N)$ does not appear in \eqref{eq:empirical1}, so the boundary generators do not contribute to the time evolution of $\rho_N$.

We prove Proposition \ref{prop:hl-balan} via two lemmas.
First, we observe that $\rho$ and $\hat\rho$ are essentially equivalent across different mesoscopic scales.

\begin{lem}
\label{lem:local-equi-balan}
If the dynamics is balanced and $K$, $K'$ satisfy that
\begin{align}
&\text{for Liggett case} \quad 1 \ll K,K' \ll \min\big\{\sqrt{N^{1+a}},N\big\};\\
&\text{for reversible case} \quad 1 \ll K,K' \ll \min\big\{\sqrt{N(N^a+\widetilde\sigma_N)\sigma_N},N\big\},
\end{align}
then for any $\vf\in\mathcal C([0,1]\times\R_+)$ and $g\in\mathcal C([0,1])$,
\begin{align}
\label{eq:local-equi-balan}
\lim_{N\to\infty} \E_N^{\mu_{N,0}} \left[ \iint_{\Sigma_T} \vf(x,t)\big[g\big(\rho_{N,K}(x,t)\big) - g\big(\hat\rho_{N,K'}(x,t)\big)\big]dx\,dt \right] = 0.
\end{align}
\end{lem}

Then we show that $\hat\rho_{N,K}$ converges to $\rho(t)$ with some specific $K=K(N)$.

\begin{lem}
\label{lem:hl-balan0}
If the dynamics is balanced and $K$ satisfies that
\begin{align}
\label{eq:assp-k-li}
&\text{for Liggett case} \quad \sqrt N \ll K \ll \min\{N^a,N\};\\
\label{eq:assp-k-ge}
&\text{for reversible case} \quad \max\big\{\sqrt N, \sigma_N\big\} \ll K \ll \min\big\{ N^a\sigma_N, \widetilde\sigma_N\sigma_N, N\big\},
\end{align}
then for any $\vf\in\mathcal C([0,1]\times\R_+)$ and $g\in\mathcal C([0,1])$,
\begin{align}
\label{eq:limit-balan0}
\lim_{N\to\infty} \iint_{\Sigma_T} \vf(x,t)g\big(\hat\rho_{N,K}(x,t)\big)dx\,dt = \iint_{\Sigma_T} \varphi(x,t)g\big(\rho(t)\big)dt
\end{align}
holds in probability for any sequence of initial distributions $\mu_{N,0}$.
\end{lem}

The assumption $a>1/2$ for Liggett case and \eqref{eq:assp-ge}, \eqref{eq:assp-a-ge} for reversible case
assure the existence of a common sequence $K=K(N)$ satisfying the conditions in Lemma \ref{lem:local-equi-balan} and \ref{lem:hl-balan0}.
Proposition \ref{prop:hl-balan} then follows straightforwardly.

\subsection{Proof of Lemma \ref{lem:local-equi-balan}}

For the balanced dynamics we can prove a stronger version of
Lemma \ref{lem:local-equi}.

\begin{proof}[Proof of Lemma \ref{lem:local-equi-balan}]
Observe that with $\bar\vf_i(t):=N\int_0^1 \chi_{i,N}(x)\vf(x,t)dx$,
\begin{align}
\label{eq:e1}
\iint_{\Sigma_T} \vf(x,t) g(\rho_N(x,t))\; dx\,dt = \frac{(2K-1)Tg(0)}{2N} + \int_0^T \frac1N\sum_{i=K}^N \bar\vf_i(t)g\big(\bar\eta_{i,K}\big)dt.
\end{align}
As $g\in\mathcal C([0,1])$ can be uniformly approximated by polynomials,
it is enough to consider $g(\rho)=\rho^\ell$ for all fixed $\ell$.
For a choice of 
$f=\eta_1\eta_2\ldots\eta_\ell$ we have $\left< f\right>(\rho) = g(\rho)$.

Fix some $K$ satisfying the conditions of Lemma \ref{lem:local-equi-balan}.
Combining Proposition \ref{prop:one-block} with \eqref{eq:dir1b} for Liggett case,
or \eqref{eq:dir2b} for reversible case, we obtain
\begin{align}
\lim_{N\to\infty} \E_N \left[ \int_0^T \frac1N\sum_{i=K}^N \big|g\big(\bar\eta_{i,K}\big)-f_{i-\ell,K-\ell}\big|^2dt \right] = 0.
\end{align}
Schwarz inequality then yields that
\begin{align}
\label{eq:e2}
\lim_{N\to\infty} \E_N \left[ \int_0^T \frac1N\sum_{i=K}^N \bar\vf_i(t)\big[g\big(\bar\eta_{i,K}\big) - f_{i-\ell,K-\ell}\big]dt \right] = 0.
\end{align}
Similarly to \eqref{eq:7-2}, $f_{i-\ell,K-\ell}$ can be replaced by $f(\tau_{i-\ell}\eta)$ with a uniformly vanishing error.
Therefore, from \eqref{eq:e1} we see that $\rho_{N,K}$ are identical for different $K$:
\begin{align}
\label{eq:e4}
\lim_{N\to\infty} \E_N \left[ \iint_{\Sigma_T} \vf g(\rho_{N,K})dx\,dt - \int_0^T \frac1N\sum_{i=\ell}^N \bar\vf_if(\tau_{i-\ell}\eta)dt \right] = 0.
\end{align}
To extend the argument to $\hat\rho_{N,K'}$, observe that $|g'(x)|=\ell x^{\ell-1}\le\ell$,
\begin{equation}
\begin{aligned}
&\E_N \left[ \int_0^T \frac1N\sum_{i=K'}^{N-K'+1} \big|g\big(\etabar_{i,K'}\big) - g\big(\bar\eta_{i+K'-1,2K'-1}\big)\big|^2dt \right] \\
\le\;&\E_N \left[ \int_0^T \frac\ell N\sum_{i=K'}^{N-K'+1} \big| \etabar_{i,K'} - \bar\eta_{i+K'-1,2K'-1}\big|^2dt \right].
\end{aligned}
\end{equation}
Using Proposition \ref{prop:one-block} and Remark \ref{rem:one-block0} with $f=\eta_1$, $k=2K'-1$ and $a_{k,j}$ be the weights in $\etabar_{i,K'}$, the right-hand side above vanishes as $N\to\infty$.
Therefore,
\begin{align}
\label{eq:e5}
\lim_{N\to\infty} \E_N \left[ \iint_{\Sigma_T} \vf g(\hat\rho_{N,K'})dx\,dt - \int_0^T \frac1N\sum_{i=\ell}^N \bar\vf_if(\tau_{i-\ell}\eta)dt \right] = 0.
\end{align}
Lemma \ref{lem:local-equi-balan} then follows from \eqref{eq:e4} and \eqref{eq:e5}.
\end{proof}

\subsection{The boundary entropy production}
\label{sec:bentropy}

We are left with the proof of Lemma \ref{lem:hl-balan0}.
Recall the boundary entropy--entropy flux pair $(\h,\q)$ in Definition \ref{bLef}.
For $\psi \in \mathcal C^1(\Sigma_T)$ and $w \in \mathcal C^1(\R_+)$,
define the \emph{boundary entropy production}
\begin{equation}\label{eq:ent-prod}
\begin{aligned}
X^\h_N(\psi,w) =\;&N^{-a}\iint_{\Sigma_T} \big[\h(\hat\rho_N,w)\partial_t \psi+ \partial_w\h(\hat\rho_N,w)w'\psi\big]dx\,dt \\
& + \iint_{\Sigma_T} \q(\hat\rho_N,w)\partial_x\psi\,dx\,dt, 
\end{aligned}
\end{equation}
where $\hat\rho_N=\hat\rho_{N,K}$ is the empirical process defined in \eqref{eq:empirical1}. 

From now on we fix an arbitrary $\rho\in\mathcal C^1([0,T])$ such that $\rho(t)\in[0,1]$. 
We see in below that balanced dynamics has zero boundary entropy production.

\begin{prop}[Liggett boundary]
\label{prop:bd-ent-li}
If $\rho_-(t)=\rho_+(t)=\rho(t)$ and \eqref{eq:assp-k-li} holds,
then
\begin{align}
\label{eq:bd-ent}
\lim_{N\to\infty} \E_N^{\mu_{N,0}} \left[ \left| X_N^F(\psi,\rho) \right| \right] = 0, 
\end{align}
for any initial distribution $\mu_{N,0}$, boundary entropy--entropy flux pair $(\h,\q)$ and $\psi \in \mathcal C^2(\Sigma_T)$ such that $\psi(\cdot,0) = \psi(\cdot,T) = 0$. 
\end{prop}

\begin{prop}[Reversible boundary]
\label{prop:bd-ent-ge}
If $\rho_-(t)=\rho_+(t)=\rho(t)$ and \eqref{eq:assp-k-ge} holds,
then the result in Proposition \ref{prop:bd-ent-li} still holds. 
\end{prop}

The proofs of Proposition \ref{prop:bd-ent-li} and \ref{prop:bd-ent-ge}
are similar and are postponed to Sections \ref{subsec:proof-li} and \ref{subsec:proof-ge}. 
From \eqref{eq:bd-ent} we can conclude the proof of Lemma \ref{lem:hl-balan0}
as follows.

\begin{proof}[Proof of Lemma \ref{lem:hl-balan0}]
Given any boundary entropy--entropy flux pair $(\h,\q)$, define the 
boundary entropy flux of a Young measure $\nu\in\mathcal Y$ 
with respect to boundary data $w \in \mathcal C([0,T])$ as the functional 
\begin{equation}
\widetilde\q(\psi;\nu,w) :=
\iint_{\Sigma_T} \psi(x,t)dx\,dt\int_0^1 \q(y,w(t))\nu_{x,t}(dy), 
\quad \forall\,\psi\in \mathcal C(\Sigma_T). 
\end{equation}
Let $\hat\nu^N$ be the Young measure associated to $\hat\rho_N$, i.e., $\hat\nu_{x,t}^N = \delta_{\hat\rho_N(x,t)}$. 
Since $\h$ and $\partial_w\h$ are bounded and $a>0$, for all $\psi\in\mathcal C^1(\Sigma_T)$, 
\begin{align}
\lim_{N\to\infty} \E_N^{\mu_{N,0}} \left[ \left| \widetilde\q(\partial_x\psi;\hat\nu_N,\rho) - X_N^\h(\psi,\rho) \right| \right] = 0, 
\end{align}
where $\rho(t)= \rho_\pm(t)$.
Since the map $\nu \mapsto \widetilde\q(\partial_x\psi;\nu,\rho)$
is a bounded linear functional on $\mathcal Y$, 
it is continuous and consequently
the set $\{\nu; |\widetilde\q(\partial_x\psi;\nu,\rho)| < \ve\}$ is open. 
Thanks to Lemma \ref{lem:local-equi-balan}, the distribution of $\hat\nu^N$ converges to $\mathfrak Q$.
By \eqref{eq:bd-ent} we have 
\begin{align}
\mathfrak Q \left( \left|\widetilde\q(\partial_x\psi;\nu,\rho)\right| > \ve \right) \le \liminf_{N\to\infty} \mathbb P_N \left( \left| X_N^F(\psi,\rho) \right| > \ve \right) = 0 
\end{align}
for any $\ve>0$ and $\psi\in \mathcal C^2(\Sigma_T)$ such that $\psi(\cdot,0) = \psi(\cdot,T) = 0$. 
Hence, the following holds with $\mathfrak Q$-probability $1$: 
\begin{align}\label{eq:bd-ent0}
\bar\q(x,t) := \int_0^1 \q\big(y,\rho(t)\big)\nu_{x,t}(dy) = 0, \quad (x,t)\:\text{-a.e. in }\Sigma_T. 
\end{align}

To prove Lemma \ref{lem:hl-balan0}, it suffices to show that $\nu_{x,t} = \delta_{\rho(t)}$ if \eqref{eq:bd-ent0} holds for all boundary entropy flux $\q$. 
We make use of the boundary entropy 
\begin{align}
\h(u,w) = 
\begin{cases} 
w\land\frac12 - u, &u \in [0, w\land\frac12), \\ 
0, &u \in [w\land\frac12, 1], 
\end{cases}
\end{align}
The corresponding boundary entropy flux is 
\begin{align}
\q(u,w) = 
\begin{cases} 
J(w\land\frac12) - J(u), &u \in [0, w\land\frac12), \\ 
0, &u \in [w\land\frac12, 1]. 
\end{cases}
\end{align}
As $\q(u,w)\ge0$ for all $(u,w)$ but $\bar\q(x,t)=0$, 
we conclude that $\nu_{x,t}$ concentrates on the zero set of $\q$,
which is $[\rho(t)\land1/2, 1]$. 
Similarly, choose 
\begin{align*}
&\h(u,w) = 
\begin{cases}
0, &u \in [0, w\lor\frac12], \\
u - w\lor\frac12, &u \in (w\lor\frac12,1], 
\end{cases}
\\
&\q(u,w) = 
\begin{cases} 
0, &u \in [0, w\lor\frac12], \\
J(u) - J(w\lor\frac12), &u \in (w\lor\frac12,1]. 
\end{cases}
\end{align*}
As $\q(u,w)\le0$. the condition $\bar\q(x,t)=0$ then implies that 
$\nu_{x,t}$ concentrates on $[0, \rho_(t)\lor1/2]$. 
Hence, $\nu_{x,t}(\La_t) = 1$ almost surely on $\Sigma_T$, where 
\begin{align}
\La_t = \left[\rho(t)\land\frac12, \rho(t)\lor\frac12\right]. 
\end{align}
Finally, to close the proof we choose 
\begin{align}
\h(u,w) = |u-w|, \quad \q(u,w) = \text{sign}(u-w)(J(u)-J(w)). 
\end{align}
If $\rho(t)<1/2$, $\q(u,\rho(t)) \ge 0$ on $\La_t = [\rho(t),1/2]$ and the only zero point is $\rho(t)$, so that $\bar\q=0$ implies $\nu_{(x,t)} = \delta_{\rho(t)}$. 
If $\rho(t)\ge1/2$ the argument is similar. 
\end{proof}

Now we prepare some notations for the proofs of Proposition \ref{prop:bd-ent-li} and \ref{prop:bd-ent-ge}.
Let
\begin{align}\label{eq:bd-set}
B_N := \left[0,\frac {2K+1}{2N}\right) \cup \left[1-\frac{2K-1}{2N},1 \right], \quad N\ge1.
\end{align}
For $\psi: \Sigma_T \to \R$ and each $i=1$, ..., $N$, 
\begin{align}
\bar\psi_i(t) := N\int_0^1 \psi(x,t) \chi_{i,N}(x)dx, \quad \tilde\psi_i(t) := \psi\left(\frac iN - \frac1{2N}, t\right). 
\end{align}
Recall the non-gradient current $J_{i,i+1}=\bar p\eta_i(1-\eta_{i+1})$ and let 
\begin{align}
\Jbar_{i,K} := \frac1{K^2}\sum_{|j|<K} \big(K-|j|\big)J_{i-j,i-j+1}.
\end{align}
We shall fix some boundary Lax entropy--entropy flux pair $(\h,\q)$ and write $X_N$ instead of $X_N^\h$ for short. 
We also omit the arbitrary initial measure $\mu_{N,0}$ and denote the expectation with respect to $\{\eta(t); t\in[0,T]\}$ by $\E_N$.

\subsection{Proof of Proposition \ref{prop:bd-ent-li}}
\label{subsec:proof-li}

We begin with a decomposition lemma.
Recall that for a sequence $\{a_i,i\in\mathbb Z\}$, $\nabla a_i=a_{i+1}-a_i$ and $\nabla^*a_i=a_{i-1}-a_i$.

\begin{lem}
\label{lem:ent-prod1}
$X_N$ satisfies the following decomposition: 
\begin{align}\label{eq:ent-prod1}
X_N(\psi,w) = M_N(\psi,w) - \sum_{i=1}^4 A_N^{(i)}(\psi,w), 
\end{align}
where $M_N$ is a square integrable martingale and $A_N^{(i)}$ are given by 
\begin{equation}\label{eq:a-li}
\begin{aligned}
A_N^{(1)} &:= \int_0^T \sum_{i=K+1}^{N-K} \bar\psi_i \partial_u\h(\etabar_{i,K},w) \nabla^* \left[ \Jbar_{i,K} - J(\etabar_{i,K}) \right]dt, \\
A_N^{(2)} &:= \frac{1-\bar p}2\int_0^T \sum_{i=K+1}^{N-K} \bar\psi_i \partial_u\h(\etabar_{i,K},w) \Delta\etabar_{i,K}\,dt, \\
A_N^{(3)} &:= \frac1N\int_0^T \sum_{i=K+1}^{N-K} \bar\psi_i \left( \ve_{i,K}^{(1)} + \ve_{i,K}^{(2)} \right)dt - \int_0^T \int_{B_N} \q(0,w)\partial_x\psi\,dx\,dt, \\
A_N^{(4)} &:= \int_0^T \sum_{i=K+1}^{N-K} \left[ \bar\psi_i \partial_u\q(\etabar_{i,K},w) \nabla^*\etabar_{i,K} - \nabla\tilde\psi_i \q(\etabar_{i,K},w) \right] dt. 
\end{aligned}
\end{equation}
Here $\ve_{i,K}^{(1)}$ and $\ve_{i,K}^{(2)}$ are respectively given by 
\begin{equation}\label{eq:eps-li}
\begin{aligned}
\ve^{(1)}_{i,K} &:= \frac N2 \sum_{j=i-K}^{i+K-1} \left( \bar p\eta_i + \frac{1-\bar p}2 \right) \partial_u^2\h(\tilde\eta_{i,j,K},w) \left( \etabar_{i,K}^{j,j+1}-\etabar_{i,K} \right)^2,
\\
\ve^{(2)}_{i,K} &:= N\big[\nabla^*J(\etabar_{i,K}) - J'(\etabar_{i,K})\nabla^* \etabar_{i,K}\big]. 
\end{aligned}
\end{equation}
where $\tilde\eta_{i,j,K}$ is some intermediate value between $\etabar_{i,K}$ and $\etabar_{i,K}^{j,j+1}$. 
\end{lem}

\begin{proof}
We omit in this proof the dependence on $w$ in $(\h,\q)$.
Observe that
\begin{align*}
\int_0^1 F\big(\hat\rho_N(x,t)\big)\psi(x,t)dx
= \frac1N\sum_{i=K+1}^{N-K} \bar\psi_i(t)F\big(\etabar_{i,K}(t)\big)
+ F(0)\int_{B_N} \psi(x,t)dx.
\end{align*}
Noting that $\psi(x,0)=\psi(x,T)=0$, Dynkin's formula yields that
\begin{equation}
\begin{aligned}
M_N(\psi,w) :=\;&N^{-a}\iint_{\Sigma_T} \big[\h(\hat\rho_N,w)\partial_t \psi + \partial_w\h(\hat\rho_N,w)w'\psi\big]dx\,dt\\
&+ N^{-a}\int_0^T \frac1N\sum_{i=K+1}^{N-K} \bar\psi_iN^{1+a}L_{N,t}[\h(\etabar_{i,K})]\,dt
\end{aligned}
\end{equation}
is a square integrable martingale.
For the second integral in \eqref{eq:ent-prod}, straightforward computation shows that it equals to
\begin{align}
\q_N(\psi,w) := \int_0^T \sum_{i=K+1}^{N-K} \q(\etabar_{i,K})\nabla\tilde\psi_i\,dt
+ \int_0^T \q(0)\int_{B_N} \partial_x\psi\,dx\,dt.
\end{align}
Therefore, \eqref{eq:ent-prod} can be rewritten as
\begin{align*}
X_N(\psi,w) = -\int_0^T \sum_{i=K+1}^{N-K}
\bar\psi_i L_{N,t}[\h(\etabar_{i,K})]\,dt + \q_N(\psi,w) + M_N(\psi,w). 
\end{align*}
Since $\hat\eta_{i,K}$ is independent of $(\eta_1,\eta_N)$ for $K+1 \le i \le N-K$,
\begin{equation}
\begin{aligned}
L_{N,t} [\h(\etabar_{i,K})] &= \sum_{j =i-K}^{i+K-1} \left( \bar p\eta_i+\frac{1-\bar p}2 \right) \left[ \h \left( \etabar^{j,j+1}_{i,K} \right) -\h(\etabar_{i,K}) \right] \\
&= \partial_u\h(\etabar_{i,K}) L_\exc [\etabar_{i,K}] + N^{-1}\ve_{i,K}^{(1)}.
\end{aligned}
\end{equation}
Recall that $L_\exc[\eta_i]=\Delta\eta_i+\nabla^*J_{i,i+1}$, $\Delta\eta_i = \nabla^*\eta_i-\nabla^*\eta_{i+1}$ and $\nabla^*J_{i,i+1}=J_{i-1,i}-J_{i,i+1}$.
Therefore, $L_\exc [\etabar_{i,K}]$ reads
\begin{equation}
\begin{aligned}
L_\exc [\etabar_{i,K}] =\;&\frac{1-\bar p}2 \Delta\etabar_{i,K} + \nabla^* \left[ \Jbar_{i,K}-J(\etabar_{i,K}) \right]\\
&+ J'(\etabar_{i,K})\nabla^*\etabar_{i,K} + N^{-1}\ve_{i,K}^{(2)}. 
\end{aligned}
\end{equation}
We then obtain that 
\begin{equation}
\begin{aligned}
X_N(\psi,w) = &\int_0^T \sum_{i=K+1}^{N-K} \left[ \bar\psi_iJ'(\etabar_{i,K})\nabla^*\etabar_{i,K} + \q(\etabar_{i,K})\nabla\tilde\psi_i \right] dt\\
&-A_N^{(1)} - A_N^{(2)} - A_N^{(3)} + M_N(\psi,w).
\end{aligned}
\end{equation}
Since $J'\partial_u\h=\partial_u\q$, the conclusion then follows.
\end{proof}
We have to evaluate each term in the right-hand side of \eqref{eq:ent-prod1}.

\begin{lem}\label{lem:m-li}
$\lim_{N\to\infty} \E_N [|M_N(\psi,\rho)|] = 0$. 
\end{lem}

\begin{proof}
The quadratic variation of $M_N$ satisfies that 
\begin{align*}
\langle M_N \rangle &= \int_0^T \sum_{j=1}^{N-1} \frac{c_{j,j+1}}{N^{1+a}}\left[\sum_{i=K+1}^{N-K} \bar\psi_i\big(\h(\etabar_{i,K}^{j,j+1}) - \h(\etabar_{i,K})\big)\right]^2dt \\
&\le \frac 1{N^{1+a}}\int_0^T \sum_{j=1}^{N-1} \left[\sum_{i=K+1}^{N-K} \bar\psi_i \partial_u\h(\tilde\eta_{i,j,K}) \big(\hat\eta_{i,K}^{j,j+1} - \hat\eta_{i,K}\big)\right]^2dt, 
\end{align*}
where $\tilde\eta_{i,j,K}$ is some intermediate value between $\etabar_{i,K}$ and $\etabar_{i,K}^{j,j+1}$. 
Direct computation shows that 
\begin{equation}
\label{eq:9}
\hat\eta_{i,K}^{j,j+1} = \hat\eta_{i,K} - \text{sgn}\left( i-j-\frac12 \right)\frac{\nabla\eta_j}{K^2}, \quad j-K+1 \le i \le j+K 
\end{equation}
and otherwise $\hat\eta_{i,K}^{j,j+1} - \hat\eta_{i,K} = 0$. 
Hence, define the block 
\begin{align}
\La_j := \{K+1 \le i \le N-K\}\cap\{j-K+1\le i \le j+K\}. 
\end{align}
Since $|\La_j| \le 2K$, we obtain from \eqref{eq:9} the estimate 
\begin{equation}
\begin{aligned}
\langle M_N \rangle &\le \frac{|\partial_uF|_\infty^2}{N^{1+a}} \int_0^T \sum_{j=1}^{N-1} \sum_{i\in\La_j} \bar\psi_i^2\sum_{i\in\La_j} \big(\hat\eta_{i,K}^{j,j+1} - \hat\eta_{i,K}\big)^2dt \\
&\le \frac{C|\partial_uF|_\infty^2}{N^{1+a}K^3} \int_0^T \sum_{j=1}^{N-1} \sum_{i\in\La_z} \bar\psi_i^2dt = \frac{C|\partial_uF|_\infty^2}{N^aK^2}\|\psi\|_{L^2(\Sigma_T)}^2. 
\end{aligned}
\end{equation}
The conclusion then follows from Doob's inequality. 
\end{proof}

For the remaining terms in \eqref{eq:ent-prod1}, we make use of the following block estimates.
They are corollaries of the one-block estimate in Proposition \ref{prop:one-block}.

\begin{cor}[One-block estimate for current]
\label{cor:one-block-li}
For balanced dynamics, 
\begin{align}
\E_N \left[ \int_0^T \sum_{i=K}^{N-K} \left[ \Jbar_{i,K}-J(\etabar_{i,K}) \right]^2dt \right] \le C \left( \frac{K^2}{N^a} + \frac NK \right), 
\end{align}
with some constant $C$ independent of $K$ or $N$. 
\end{cor}

\begin{cor}[$H_1$ estimate]
\label{cor:h1-li}
For balanced dynamics,
\begin{align}
\E_N \left[ \int_0^T \sum_{i=K}^{N-K} (\nabla\etabar_{i,K})^2dt \right] \le C \left( \frac1{N^a} + \frac N{K^3} \right), 
\end{align}
with some constant $C$ independent of $K$ or $N$. 
\end{cor}

\begin{proof}[Proof of Corollary \ref{cor:one-block-li}]
Similarly to Lemma \ref{lem:local-equi-balan}, take $\ell=2$, $f=\bar p\eta_1(1-\eta_2)$, $k=2K$ and observe that the weighted average $\Jbar_{i-K,K}=f_{i-2,2K-2}^*$, where 
\begin{align}
f_{i,k}^* := \sum_{j=0}^{2K} a_{2K,j}f(\tau_{i-j}\eta), \quad a_{2K,j}=\frac{K+1-|j-K|}{(K+1)^2}.
\end{align}
Since $\langle f \rangle(\rho)=\bar p\rho(1-\rho)=J(\rho)$, by Remark \ref{rem:one-block0} and Proposition \ref{prop:one-block},
\begin{align}
\label{eq:hatbar1}
\int \sum_{i=2K}^N \left[ \Jbar_{i-K,K}-J\big(\bar\eta_{i,2K}\big) \right]^2d\mu_{N,t} \le C \left[ K^2\D_{\exc,N}(t) + \frac NK \right]. 
\end{align}
Similarly, take $\ell=1$, $f=\eta_1$, $k=2K-2$ and 
\begin{align}
f_{i,k}^* := \sum_{j=0}^{2K-2} \frac{K-|j-K+1|}{K^2}f(\tau_{i-j}\eta),
\end{align}
then $f_{i-1,k-1}^*=\etabar_{i-K+1,K}$ and the same argument gives that
\begin{align}
\label{eq:hatbar2}
\int \sum_{i=2K-1}^N \big(\etabar_{i-K+1,K}-\bar\eta_{i,2K-1}\big)^2d\mu_{N,t} \le C' \left[ K^2\D_{\exc,N}(t) + \frac NK \right]. 
\end{align}
As $|\bar\eta_{i,2K}-\bar\eta_{i,2K-1}| \le K^{-1}$ and $J'$ is bounded,
the corollary follows from \eqref{eq:hatbar1}, \eqref{eq:hatbar2} and Proposition \ref{prop:dir1}.
\end{proof}

\begin{proof}[Proof of Corollary \ref{cor:h1-li}]
Observe that for $i=K$, $K+1$, ..., $N-K$, 
\begin{align}\label{eq:barhat}
\nabla\etabar_{i,K} = \frac{\bar\eta_{i+K,K}-\bar\eta_{i,K}}{K} = \frac2K \left( \frac{\eta_{i+1}}K+\ldots+\frac{\eta_{i+K}}K-\bar\eta_{i+K,2K} \right). 
\end{align}
Proposition \ref{prop:one-block} and Remark \ref{rem:one-block0} then yield that
\begin{equation}
\begin{aligned}
&\E_N \left[ \int_0^T \sum_{i=K}^{N-K} (\nabla\etabar_{i,K})^2dt \right]\\
=\;&\E_N \left[ \frac4{K^2}\int_0^T \sum_{i=2K}^N \left| \frac{\eta_{i-K+1}}K+\ldots+\frac{\eta_i}K - \bar\eta_{i,2K} \right|^2 dt \right]\\
\le\;&\frac C{K^2} \left[ K^2\int_0^T \D_{\exc,N}(t)dt + \frac{NT}K \right] \le C' \left( \frac1{N^a} + \frac N{K^3} \right),
\end{aligned}
\end{equation}
where the last line follows from Proposition \ref{prop:dir1}.
\end{proof}

The following result helps us treat the blocks located at the boundaries. 

\begin{prop}
\label{prop:bd-block-li}
For balanced dynamics with $\rho_-(t)=\rho_+(t)=\rho(t)$, 
\begin{equation}
\begin{aligned}
\E_N \left[ \int_0^T \big|\etabar_{K,K}-\rho(t)\big|^2dt \right] &\le C \left( \frac K{N^a} + \frac 1K \right), \\
\E_N \left[ \int_0^T \big|\nabla\etabar_{K,K}\big|^2dt \right] &\le C \left( \frac 1{N^aK} + \frac 1{K^3} \right), \\
\end{aligned}
\end{equation}
with some constant $C$ independent of $K$ or $N$.
The same upper bounds hold with $\etabar_{N-K+1,N}$ and $\nabla\etabar_{N-K}$.
\end{prop}

\begin{proof}[Proof of Proposition \ref{prop:bd-block-li}]
The proof goes similarly to Proposition \ref{prop:one-block}.
Denote by $\mu_{N,t}^K$ the distribution of $\{\eta_1, \ldots, \eta_K\}$ at time $t$ and let $f_{N,t}^K = \mu_{N,t}^K/\nu_{\rho(t)}$ be the density with respect to the Bernoulli measure.
By the relative entropy inequality, 
\begin{align*}
\int \big|\etabar_{K+1,K}-\rho(t)\big|^2d\mu_{N,t} \le \frac1K\left[ H\big(\mu_{N,t}^K;\nu_{\rho(t)}\big) + \log\int e^{K|\etabar_{K,K}-\rho(t)|^2}d\nu_{\rho(t)} \right].
\end{align*}
Applying Proposition \ref{prop:logsob-bound} proved in Appendix \ref{sec:log-sobol-ineq}, 
\begin{align*}
H\big(\mu_{N,t}^K;\nu_{\rho(t)}\big) \le \frac{CK^2}2\sum_{\eta\in\Om_K} \sum_{j=1}^{K-1} \left( \sqrt{f_{N,t}^K(\eta^{j,j+1})}-\sqrt{f_{N,t}^K(\eta)} \right)^2 \nu_{\rho(t)}(\eta) \\
+ \frac{CK}2\sum_{\eta\in\Om_K} \rho(t)^{1-\eta_1}(1-\rho(t))^{\eta_1} \left( \sqrt{f_{N,t}^K(\eta^1)}-\sqrt{f_{N,t}^K(\eta)} \right)^2 \nu_{\rho(t)}(\eta), 
\end{align*}
with some universal constant $C$. Therefore, 
\begin{align}
H\big(\mu_{N,t}^K;\nu_{\rho(t)}\big) \le CK^2\D_{\exc,N}(t) + CK\D_{-,N}(t). 
\end{align}
As the exponential moment with respect to $\nu_{\rho(t)}$ is uniformly bounded,
\begin{align}
\int \big|\etabar_{K+1,K}-\rho(t)\big|^2d\mu_{N,t} \le C \left[ K\D_{\exc,N}(t) + \D_{-,N}(t) + \frac1K \right].
\end{align}
We only need to integrate in time and apply \eqref{eq:dir1b}.
The other estimates can be proved in the same way.
\end{proof}


\begin{rem}
From Proposition \ref{prop:dir1} and the proofs below, the factor $N^{-a}$ in the previous estimates is available only for balanced dynamics.
For unbalanced dynamics, these estimates hold with $N^{-a}$ replaced by $1$. 
\end{rem}

Now we bound each term in \eqref{eq:a-li} for balanced dynamics.

\begin{lem}\label{lem:a1-li}
Assume $a>1/2$ and \eqref{eq:assp-k-li}, then 
\begin{align}
\lim_{N\to\infty} \E_N \left[ \big|A_N^{(1)}(\psi,\rho)\big|^2 \right] = 0. 
\end{align}
\end{lem}

\begin{proof}
By summation by parts and the intermediate value theorem, 
\begin{align*}
&A_N^{(1)}(\psi,\rho) = A_N^{(1,1)}+A_N^{(1,2)}+A_N^{(1,-)}-A_N^{(1,+)}, \\
&A_N^{(1,1)} = \int_0^T \sum_{i=K+1}^{N-K} \nabla\bar\psi_i \partial_u\h\big(\etabar_{i+1,K},\rho(t)\big) \left[ \Jbar_{i,K}-J(\etabar_{i,K}) \right]dt, \\
&A_N^{(1,2)} = \int_0^T \sum_{i=K+1}^{N-K} \bar\psi_i \partial_u^2\h\big(\xi_{i,K},\rho(t)\big) \nabla\etabar_{i,K} \left[ \Jbar_{i,K}-J(\etabar_{i,K}) \right]dt, \\
&A_N^{(1,-)} = \int_0^T \bar\psi_{K+1} \partial_u\h\big(\etabar_{K+1,K},\rho(t)\big) \left[ \Jbar_{K,K}-J(\etabar_{K,K}) \right]dt, \\
&A_N^{(1,+)} = \int_0^T \bar\psi_{N-K+1} \partial_u\h\big(\etabar_{N-K+1,K},\rho(t)\big) \left[ \Jbar_{N-K,K}-J(\etabar_{N-K,K}) \right]dt, 
\end{align*}
where $\xi_{i,K}$ is some intermediate value between $\etabar_{i,K}$ and $\etabar_{x+1,K}$. 
By Corollary \ref{cor:one-block-li}, 
\begin{align}
\E_N \left[ \big|A_N^{(1,1)}\big|^2 \right] \le C|\partial_x\psi|_\infty^2|\partial_u\h|_\infty^2 \left( \frac{K^2}{N^{1+a}} + \frac1K \right). 
\end{align}
For $A_N^{(1,2)}$, using Schwarz inequality, Corollary \ref{cor:one-block-li} and \ref{cor:h1-li}, 
\begin{align}
\E_N \left[ \big|A_N^{(1,2)}\big|^2 \right] &\le C|\psi|_\infty^2 |\partial_u^2\h|_\infty^2 \left( \frac{K^2}{N^{2a}} + \frac{N^2}{K^4} \right). 
\end{align}
For the boundary terms, recall that $\partial_uF(u,w)|_{u=w} \equiv 0$, then 
\begin{align}\label{eq:boundterm}
\partial_u\h(\etabar_{K+1,K},\rho(t)) = \partial_u^2\h(\eta_-, \rho(t))\big(\etabar_{K+1,K}-\rho(t)\big), 
\end{align}
for some intermediate value $\eta_-$ between $\etabar_{K+1,K}$ and $\rho(t)$. 
Hence, 
\begin{align}
\E_N \left[ \big|A_N^{(1,-)}\big|^2 \right] \le C|\psi|_\infty^2|\partial_u^2\h|_\infty^2 \left( \frac K{N^a} + \frac 1K \right), 
\end{align}
thanks to Proposition \ref{prop:bd-block-li} and the boundedness of $J$. 
The right boundary term $A_N^{(1,+)}$ can be estimated similarly. 
When $N\to\infty$, all the upper bounds vanish since $K$ is chosen to satisfy \eqref{eq:assp-k-li}. 
\end{proof}

\begin{lem}\label{lem:a2-li}
Assume $a>0$ and $K \gg N^{1/3}$, then 
\begin{align}
\lim_{N\to\infty} \E_N \left[ \big|A_N^{(2)}(\psi,\rho)\big| \right] = 0. 
\end{align}
\end{lem}

\begin{proof}
Similarly to $A_N^{(1)}$, with some $\xi_{i,K}$ between $\etabar_{i,K}$ and $\etabar_{i+1,K}$, 
\begin{align*}
&A_N^{(2)}(\psi,\rho) = A_N^{(2,1)} + A_N^{(2,2)} + A_N^{(2,-)} + A_N^{(2,+)}, \\
&A_N^{(2,1)} = - \frac{1-\bar p}2 \int_0^T \sum_{i=K}^{N-K}
\nabla\bar\psi_i \partial_u\h\big(\etabar_{i+1,K},\rho(t)\big)
\nabla\etabar_{i,K}\,dt, \\
&A_N^{(2,2)} = - \frac{1-\bar p}2 \int_0^T \sum_{i=K}^{N-K} 
\bar\psi_i \partial_u^2\h\big(\xi_{i,K},\rho(t)\big) 
\big(\nabla\etabar_{i,K}\big)^2dt, \\
&A_N^{(2,-)} = - \frac{1-\bar p}2 \int_0^T 
\bar\psi_K \partial_u\h\big(\etabar_{K,K},\rho(t)\big)
\nabla\etabar_{K+1,K}\,dt, \\
&A_N^{(2,+)} = \frac{1-\bar p}2 \int_0^T 
\bar\psi_{N-K+1} \partial_u\h\big(\etabar_{N-K+1,K},\rho(t)\big) 
\nabla\etabar_{N-K,K}\,dt. 
\end{align*}
Due to the $H_1$ estimate in Corollary \ref{cor:h1-li}, 
\begin{align}
\E_N \left[ \big|A_N^{(2,1)}\big|^2 + \big|A_N^{(2,2)}\big| \right] &\le C(\psi,\h) \left( \frac1N + 1 \right) \left( \frac1{N^a} + \frac N{K^3} \right). 
\end{align}
For the boundary terms, similarly to \eqref{eq:boundterm}, 
\begin{equation}
\begin{aligned}
\E_N \left[ \big|A_N^{(2,-)}\big|^2 \right] \le\:&C|\psi|_\infty^2| \partial_u^2\h|_\infty^2 \times \E_N \left[ \int_0^T \big(\nabla\etabar_{K+1,K}\big)^2dt \right] \\
&\times \E_N \left[ \int_0^T \big(\etabar_{K,K}-\rho(t)\big)^2dt \right] \\
\le\:&C'|\psi|_\infty^2|\partial_u^2\h|_\infty^2 \left( \frac1{N^{2a}} + \frac1{K^4} \right), 
\end{aligned}
\end{equation}
where the last line follows from Proposition \ref{prop:bd-block-li}. 
The last term is bounded similarly. 
Observe that all bounds vanish under our conditions. 
\end{proof}

\begin{lem}\label{lem:a3-li}
Assume $a>1/2$ and \eqref{eq:assp-k-li}, then $A_N^{(3)}(\psi,\rho)\to0$ uniformly. 
\end{lem}

\begin{proof}
Observe from \eqref{eq:eps-li} and \eqref{eq:9} that for any $i$, 
\begin{align}
\lim_{N\to\infty} \left| \ve_{i,K}^{(1)} \right| \le \lim_{N\to\infty} \frac{CN|\partial_u^2\h|_\infty}{K^3} = 0. 
\end{align}
Meanwhile, noting that $J=\bar p\rho(1-\rho)$ and $J''=-2\bar p$, we obtain that 
\begin{equation}
\begin{aligned}
\left| \ve_{i,K}^{(2)} \right| &= N\big|J'(c\etabar_{i-1,K} + (1-c)\etabar_{i,K}) - J'(\etabar_{i,K})\big| \big|\nabla^*\etabar_{i,K}\big| \\
&= 2N\bar p(1-c) \big|\nabla^*\etabar_{i,K}\big|^2 \le \frac{CN}{K^2}, 
\end{aligned}
\end{equation}
with some $\xi\in[0,1]$. 
Therefore, they vanish uniformly as $N\to\infty$. 

We are left with the integral with respect to $B_N$. 
Recall the definition of $B_N$ in \eqref{eq:bd-set} and note that it has Lebesgue measure $2K/N$, so that 
\begin{align}
\left| \int_0^T \int_{B_N} \q(0,\rho)\partial_x\psi\,dx\,dt \right| \le \frac{C|\partial_x\psi|_\infty|\q|_\infty K}N. 
\end{align}
Thus, this term also vanishes uniformly as $N\to\infty$. 
\end{proof}

\begin{lem}\label{lem:a4-li}
Assume $a>1/2$ and \eqref{eq:assp-k-li}, then 
\begin{align}
\lim_{N \to \infty} \E_N \left[ \big|A_N^{(4)}(\psi,\rho)\big| \right] = 0. 
\end{align}
\end{lem}

\begin{proof}
Similarly to $A_N^{(2)}$, with some $\xi_{i,K}$ between $\etabar_{i,K}$ and $\etabar_{i+1,K}$, 
\begin{align*}
&A_N^{(4)}(\psi,\rho) = A_N^{(4,1)} + A_N^{(4,2)} + A_N^{(4,\text{bd})}, \\
&A_N^{(4,1)} = \int_0^T \sum_{i=K+1}^{N-K} \big(\bar\psi_i-\tilde\psi_i\big) \partial_u\q\big(\etabar_{i,K},\rho(t)\big) \nabla^*\etabar_{i,K}\,dt, \\
&A_N^{(4,2)} = - \int_0^T \sum_{i=K+1}^{N-K} \tilde\psi_i \big[\partial_u\q(\etabar_{i,K},\rho) \nabla^*\etabar_{i,K} - \nabla^*\q(\etabar_{i,K},\rho)\big]dt, \\
&A_N^{(4,\text{bd})} = \int_0^T \tilde\psi_{K+1}\q\big(\etabar_{K,K},\rho(t)\big)dt - \int_0^T \tilde\psi_{N-K+1}\q\big(\etabar_{N-K,K},\rho(t)\big)dt. 
\end{align*}
For $A_N^{(4,1)}$, direct calculation shows that $|\bar\psi_i - \tilde\psi_i| \le C|\partial_x\psi|_\infty N^{-1}$, so that $|A_N^{(4,1)}| \le C(\psi,\q)K^{-1}$. 
Meanwhile, $|A_N^{(4,2)}| \le C(\psi,Q)NK^{-2}$ because 
\begin{align}
\big|\partial_u\q(\etabar_{i,K},\rho) \nabla^*\etabar_{i,K} - \nabla^*\q(\etabar_{i,K},\rho)\big| \le |\partial_u^2\q|_\infty \big|\nabla^*\etabar_{i,K}\big|^2. 
\end{align}
Therefore, these two terms vanish uniformly if $K^2 \gg N$. 

We are left with the boundary term. 
Recalling that $\q(w,w) \equiv 0$ for all $w\in\R$, we have $|\q(\etabar_{K,K},\rho(t))| \le |\partial_u\q|_\infty|\etabar_{K,K}-\rho(t)|$. 
Since similar estimate holds for $Q(\etabar_{N-K,K},\rho(t))$, in view of Proposition \ref{prop:bd-block-li}, 
\begin{align}
\E_N \left[ \big|A_N^{(4,\text{bd})}\big|^2 \right] \le C|\psi|_\infty^2|\partial_u\q|_\infty^2 \left( \frac K{N^a} + \frac1K \right). 
\end{align}
The desired estimate then follows from \eqref{eq:assp-k-li}. 
\end{proof}

\subsection{Proof of Proposition \ref{prop:bd-ent-ge}}
\label{subsec:proof-ge}

As the proof for reversible case goes parallel to that of Liggett case,
we only emphasize the difference here.

By the same computation as in Lemma \ref{lem:ent-prod1}, $X_N$ satisfies the decomposition formula \eqref{eq:ent-prod1}, where $A_N^{(i)}$, $i=1$, $3$, $4$ and $\ve_{i,K}^{(2)}$ are given in \eqref{eq:a-li}, \eqref{eq:eps-li}, 
\begin{align*}
A_N^{(2)} &:= \frac{\sigma_N-\bar p}2\int_0^T \sum_{i=K+1}^{N-K} \bar\psi_i \partial_u\h(\etabar_{i,K},w) \Delta\etabar_{i,K}\,dt, \\
\ve^{(1)}_{i,K} &:= \frac N2 \sum_{j=i-K}^{i+K-1} \left( \bar p\eta_i + \frac{\sigma_N-\bar p}2 \right) \partial_u^2\h(\tilde\eta_{i,j,K},w) \left( \etabar_{i,K}^{j,j+1}-\etabar_{i,K} \right)^2, 
\end{align*}
with proper intermediate value $\tilde\eta_{i,j,K}$ between $\etabar_{i,K}$ and $\etabar_{i,K}^{j,j+1}$. 

To continue, we make use of the following block estimates. 
Observe that they differ from those obtained for Liggett boundaries, since the Dirichlet forms possess different upper bounds here (Proposition \ref{prop:dir2}):
\begin{align*}
\int_0^T \D_{\exc,N}(t)dt \le \frac C{\sigma_N} \left( \frac1{N^a}+\frac1{\widetilde\sigma_N} \right), \quad \int_0^T \D_{\pm,N}(t)dt \le \frac C{\widetilde\sigma_N} \left( \frac1{N^a}+\frac1{\widetilde\sigma_N} \right). 
\end{align*}
Their proofs are same as the Liggett case, so we omit them.

\begin{cor}[One-block estimate for current]
\label{cor:one-block-ge}
For balanced dynamics,
\begin{align}
\E_N \left[ \int_0^T \sum_{i=K}^{N-K} \left[ \Jbar_{i,K}-J(\etabar_{i,K}) \right]^2dt \right] \le C\left[ \frac{K^2}{\sigma_N} \left( \frac1{N^a} + \frac1{\widetilde\sigma_N} \right) + \frac NK \right], 
\end{align}
with some constant $C$ independent of $K$ or $N$. 
\end{cor}

\begin{cor}[$H_1$ estimate]
\label{cor:h1-ge}
For balanced dynamics,
\begin{align}
\E_N \left[ \int_0^T \sum_{i=K}^{N-K} (\nabla\etabar_{i,K})^2dt \right] \le C \left[ \frac1{\sigma_N} \left( \frac1{N^a} + \frac1{\widetilde\sigma_N} \right) + \frac N{K^3} \right], 
\end{align}
with some constant $C$ independent of $K$ or $N$. 
\end{cor}

\begin{prop}
\label{prop:bd-block-ge}
For balanced dynamics with $\rho_-(t)=\rho_+(t)=\rho(t)$,
\begin{equation}
\begin{aligned}
\E_N \left[ \int_0^T \big|\etabar_{K,K}-\rho(t)\big|^2dt \right] &\le C \left( \frac K{\sigma_N} + \frac1{\widetilde\sigma_N} \right) \left( \frac1{N^a} + \frac1{\widetilde\sigma_N} \right) + \frac CK,\\
\E_N \left[ \int_0^T \big|\nabla\etabar_{K,K}\big|^2dt \right] &\le \frac C{K\sigma_N} \left( \frac1{N^a} + \frac1{\widetilde\sigma_N} \right) + \frac C{K^3},
\end{aligned}
\end{equation}
with some constant $C$ independent of $K$ or $N$.
The same upper bounds hold with $\etabar_{N-K+1,N}$ and $\nabla\etabar_{N-K}$.
\end{prop}

To show Proposition \ref{prop:bd-ent-ge}, it suffices to evaluate each term in the decomposition \eqref{eq:ent-prod1}. 
We sketch the proof in two lemmas. 

\begin{lem}
Assume \eqref{eq:assp-ge}, \eqref{eq:assp-a-ge} and \eqref{eq:assp-k-ge}, then 
\begin{align}
\lim_{N\to\infty} \E_N \left[ |M_N| + \big|A_N^{(1)}\big|^2 + \big|A_N^{(3)}\big|^2 + \big|A_N^{(4)}\big|^2 \right] = 0. 
\end{align}
\end{lem}

\begin{proof}
For the martingale $M_N$, its quadratic variation $\langle M_N \rangle$ reads 
\begin{align}
&\frac1{N^{1+a}}\int_0^T \sum_{j=1}^{N-1} \left( \bar p\eta_i+\frac{\sigma_N-\bar p}2 \right) \left[ \sum_{i=K+1}^{N-K} \bar\psi_i\big(\h(\etabar_{i,K}^{j,j+1}) - \h(\etabar_{i,K})\big) \right]^2dt. 
\end{align}
Using \eqref{eq:9} and the same argument as in proving Lemma \ref{lem:m-li}, 
\begin{align}
\langle M_N \rangle \le \frac{C(\sigma_N+\bar p)}{N^{1+a}K} \int_0^T \sum_{j=1}^{N-1} \bar\psi_i^2dt = \frac{C(\sigma_N+\bar p)}{N^aK^2}\|\psi\|_{L^2(\Sigma_T)}^2. 
\end{align}
For $A_N^{(1)}$, applying Corollary \ref{cor:one-block-ge}, \ref{cor:h1-ge} and the argument used in proving Lemma \ref{lem:a1-li},
\begin{align*}
\E_N \left[ \big|A_N^{(1)}\big|^2 \right] \le\:&C_1(\psi,\h) \left[ \frac{K^2}{N\sigma_N} \left( \frac1{N^a} + \frac1{\widetilde\sigma_N} \right) + \frac1K \right] \\
&+ C_2(\psi,\h) \left( \frac{K^2}{N^{2a}\sigma_N^2}+\frac{K^2}{\sigma_N^2\widetilde\sigma_N^2}+\frac{N^2}{K^4} \right) \\
&+ C_3(\psi,\h) \left[ \left( \frac K{\sigma_N} + \frac1{\widetilde\sigma_N} \right) \left( \frac1{N^a} + \frac1{\widetilde\sigma_N} \right) + \frac 1K \right]. 
\end{align*}
For $A_N^{(3)}$, it vanishes uniformly as 
\begin{align}
\left| \ve_{i,K}^{(1)} \right| \le \frac{CN\sigma_N}{K^3}, \quad \left| \ve_{i,K}^{(2)} \right| \le \frac{CN}{K^2}, \quad \left| \int_{B_N} dx \right| \le \frac{CK}N. 
\end{align}
For $A_N^{(4)}$, we can argue similarly to Lemma \ref{lem:a4-li} to obtain that 
\begin{align*}
\E_N \left[ \big|A_N^{(4)}\big|^2 \right] \le\:&C(\psi,\q) \left( \frac1K + \frac N{K^2} \right) \\
&+ C'(\psi,\q) \left[ \left( \frac K{\sigma_N} + \frac1{\widetilde\sigma_N} \right) \left( \frac1{N^a} + \frac1{\widetilde\sigma_N} \right) + \frac1K \right]. 
\end{align*}
Thanks to \eqref{eq:assp-ge} and \eqref{eq:assp-k-ge}, we have as $N\to\infty$, 
\begin{align}
\frac KN = o(1), \quad \frac K{\sigma_N} \left( \frac1{N^a} + \frac1{\widetilde\sigma_N} \right) = o(1), \quad \frac N{K^2} = o(1), \quad \frac{N\sigma_N}{K^3} = o(1), 
\end{align}
which assures the vanishing of all the bounds above. 
\end{proof}

\begin{lem}
Assume \eqref{eq:assp-ge}, \eqref{eq:assp-a-ge} and \eqref{eq:assp-k-ge}, then 
\begin{align}
\lim_{N\to\infty} \E_N \left[ \big|A_N^{(2)}(\psi,\rho)\big| \right] = 0. 
\end{align}
\end{lem}

\begin{proof}
With some $\xi_{i,K}$ between $\etabar_{i,K}$ and $\etabar_{i+1,K}$ we have 
\begin{align*}
&A_N^{(2)}(\psi,\rho) = A_N^{(2,1)} + A_N^{(2,2)} + A_N^{(2,-)} + A_N^{(2,+)}, \\
&A_N^{(2,1)} = - \frac{\sigma_N-\bar p}2 \int_0^T \sum_{i=K}^{N-K}
\nabla\bar\psi_i \partial_u\h\big(\etabar_{i+1,K},\rho(t)\big)
\nabla\etabar_{i,K}\,dt, \\
&A_N^{(2,2)} = - \frac{\sigma_N-\bar p}2 \int_0^T \sum_{i=K}^{N-K} 
\bar\psi_i \partial_u^2\h\big(\xi_{i,K},\rho(t)\big) 
\big(\nabla\etabar_{i,K}\big)^2dt, \\
&A_N^{(2,-)} = - \frac{\sigma_N-\bar p}2 \int_0^T 
\bar\psi_K \partial_u\h\big(\etabar_{K,K},\rho(t)\big)
\nabla\etabar_{K+1,K}\,dt, \\
&A_N^{(2,+)} = \frac{\sigma_N-\bar p}2 \int_0^T 
\bar\psi_{N-K+1} \partial_u\h\big(\etabar_{N-K+1,K},\rho(t)\big) 
\nabla\etabar_{N-K,K}\,dt. 
\end{align*}
By the $H_1$ estimate in Corollary \ref{cor:h1-ge}, as $\sigma_N \ll N$, 
\begin{align*}
\E_N \left[ \big|A_N^{(2,1)}\big|^2 + \big|A_N^{(2,2)}\big| \right] &\le C(\psi,\h) \left( \frac{\sigma_N^2}N + \sigma_N \right) \left( \frac1{N^a\sigma_N} + \frac1{\widetilde\sigma_N\sigma_N} + \frac N{K^3} \right) \\
&\le C'(\psi,\h)\left(\frac1{N^a} + \frac1{\widetilde\sigma_N} + \frac{N\sigma_N}{K^3} \right). 
\end{align*}
We are left with the boundary terms. 
Similarly to \eqref{eq:boundterm}, 
\begin{equation}
\begin{aligned}
\E_N \left[ \big|A_N^{(2,-)}\big|^2 \right] \le\:&C(\psi,\h) \times \sigma_N^2 \times \E_N \left[ \int_0^T \big(\nabla\etabar_{K,K}\big)^2dt \right] \\
&\times \E_N \left[ \int_0^T \big(\etabar_{K,K}-\rho_-(t)\big)^2dt \right] \\
\le\:&C' \left( 1 + \frac{\sigma_N}{K\widetilde\sigma_N} \right) \left( \frac1{N^{2a}}+\frac1{\widetilde\sigma_N^2} \right) + \frac{\sigma_N^2}{K^4}, 
\end{aligned}
\end{equation}
where the last line follows from Proposition \ref{prop:bd-block-ge}. 
The right boundary term is estimated similarly. 
Finally, the proof is completed by noting that all the bounds above vanish as $N\to\infty$ under our conditions. 
\end{proof}

\begin{rem}
From the proof above we see that the expectation of $A_N^{(2)}$ does not vanish if $\rho_-\not=\rho_+$. 
Hence, it is responsible for the non-zero entropy production associated to the solution of \eqref{eq:qscl} in this case. 
\end{rem}

\section{Unbalanced dynamics: coupling}
\label{sec:coupling}

In this section we prove Lemma \ref{lem:nu} and \ref{lem:j} by a coupling argument. 
Recall that in Liggett case, the time-dependent boundary rates are given by 
\begin{align}
\label{eq:rates1}
(\alpha,\beta,\gamma,\delta) = \left( \frac{1+\bar p}2\rho_-, \frac{1+\bar p}2(1-\rho_+), \frac{1-\bar p}2(1-\rho_-), \frac{1-\bar p}2\rho_+ \right), 
\end{align}
while in reversible case, they are
\begin{align}
\label{eq:rates2}
(\alpha,\beta,\gamma,\delta) = \widetilde\sigma_N
\big(\la_-\rho_-, \la_+(1- \rho_+), \la_-(1-\rho_-), \la_+\rho_+\big). 
\end{align}
Recall the limit distribution $\mathfrak Q$ on $\mathcal Y$ satisfying \eqref{eq:33} and the current $\mathcal J(T)$ satisfying \eqref{eq:quasi-current}, both associated with the boundary rates $(\alpha,\beta,\gamma,\delta)$. 

\begin{lem}\label{lem:mono-nu}
If $\alpha\le\alpha_*$, $\gamma\ge\gamma_*$ on $[0,T]$, then for each $y\in[0,1]$, 
\begin{align}
E^{\mathfrak Q} \big[\nu_{x,t}([y,1])\big] \le E^{\mathfrak Q_*} \big[\nu_{x,t}([y,1])\big], \quad (x,t)-\text{a.e. in }\Sigma_T, 
\end{align}
where $\mathfrak Q_*$ is the Young measure corresponding to $(\alpha_*,\beta,\gamma_*,\delta)$. 
The same result holds for $(\alpha,\beta_*,\gamma,\delta_*)$ such that $\beta\ge\beta_*$ and $\delta\le\delta_*$ on $[0,T]$. 
\end{lem}

\begin{lem}\label{lem:mono-j}
Suppose that $\alpha\le\alpha_*$, $\gamma\ge\gamma_*$ on $[0,T]$. 
By $\mathcal J_*(T)$ we denote the current associated with the boundary rates $(\alpha_*,\beta,\gamma_*,\delta)$, then $\mathcal J(T) \le \mathcal J_*(T)$.
The inverse inequality holds for the current associated with $(\alpha,\beta_*,\gamma,\delta_*)$ if $\beta\ge\beta_*$, $\delta\le\delta_*$ on $[0,T]$. 
\end{lem}

The proofs of Lemma \ref{lem:mono-nu} and \ref{lem:mono-j} are postponed to the end of this section. 
We here first show Lemma \ref{lem:nu} and \ref{lem:j} based on them. 

\begin{proof}[Proof of Lemma \ref{lem:nu}]
In Liggett case, let $\rho'(t) = \max\{\rho_-(t),\rho_+(t)\}$ for $t\in[0,T]$. 
Define $(\alpha_*, \beta_*, \gamma_*, \delta_*)$ through \eqref{eq:rates1} with $\rho_-=\rho_+=\rho'$, 
then $\alpha\le\alpha_*$, $\gamma\ge\gamma_*$, $\beta\ge\beta_*$, $\delta\le\delta_*$. 
Denote by $\mathfrak Q_*$ the limit point associated to $(\alpha_*,\beta_*,\gamma_*,\delta_*)$. 
Lemma \ref{prop:hl-balan} yields that $\mathfrak Q_*$ concentrates on the Young measure $\nu_{x,t} = \delta_{\rho_*(t)}$. 
Applying Lemma \ref{lem:mono-nu}, 
\begin{align}
E^{\mathfrak Q} \big[\nu_{x,t}([y,1])\big] \le E^{\mathfrak Q_*} \big[\nu_{x,t}([y,1])\big] = \1\{y \le \rho'(t)\}. 
\end{align}
Similarly, $E^{\mathfrak Q} [\nu_{x,t}([y,1])] \ge \1\{y \le \min(\rho_-,\rho_+)\}$. 
The proof is then completed for Liggett boundaries. 
The reversible case can be proved in exactly the same way. 
\end{proof}

\begin{proof}[Proof of Lemma \ref{lem:j}]
As before we prove only the Liggett case. 
Suppose that $\rho_-(t)\le\rho_+(t)$ for $t\in[0,T]$ 
and let $\rho_*(t) \in [\rho_-(t),\rho_+(t)]$ such that 
\begin{align}
J(\rho_*(t)) = \inf \big\{J(\rho); \rho\in[\rho_-(t),\rho_+(t)]\big\}. 
\end{align}
Observe that $\rho_*$ may not be unique when $\rho_-+\rho_+=1$.
Define $(\alpha_*, \beta_*, \gamma_*, \delta_*)$ through \eqref{eq:rates1} with $\rho_-=\rho_+=\rho_*$. 
Since $\alpha\le\alpha_*$, $\gamma\ge\gamma_*$, $\beta\le\beta_*$, $\delta\ge\delta_*$, thanks to Lemma \ref{lem:mono-j} and Proposition \ref{prop:hl-balan}, 
\begin{align}
  \mathcal J(T) \le \mathcal J_*(T) = \int_0^T J(\rho_*(t))dt. 
\end{align}
The criteria \eqref{eq:j2} then follows.
The other one is proved similarly. 
\end{proof}

Both Lemma \ref{lem:mono-nu} and Lemma \ref{lem:mono-j} are consequences of the so-called standard coupling for simple exclusion process. 
To construct the coupling, define $\bar\Om_N := \{\xi=\eta \oplus \eta'; \eta_i \le \eta'_i, \forall\:i=1,\ldots,N\}$. 
For $\xi \in \bar\Om_N$, let 
\begin{equation}
\begin{aligned}
&\xi^{1,+} := \eta^{1,+} \oplus (\eta')^{1,+}, &&\xi^{1,-} := \eta^{1,-} \oplus (\eta')^{1,-}, \\
&\xi^{N,+} := \eta^{N,+} \oplus (\eta')^{N,+}, &&\xi^{N,-} := \eta^{N,-} \oplus (\eta')^{N,-}, \\
&\xi^{N,*} := \eta^{N,-} \oplus (\eta')^{N,+}, &&\xi^{x,x+1} := \eta^{i,i+1} \oplus (\eta')^{i,i+1}, 
\end{aligned}
\end{equation}
where for $\eta \in \Om_N$, $\eta^{1,\pm}$ are $\eta^{N,\pm}$ are obtained through 
\begin{equation}
\begin{aligned}
&\eta^{1,+} := (1,\eta_2,\ldots,\eta_N), &&\eta^{1,-} := (0,\eta_2,\ldots,\eta_N), \\
&\eta^{N,+} := (\eta_1,\ldots,\eta_{N-1},1), &&\eta^{N,-} := (\eta_1,\ldots,\eta_{N-1},0). 
\end{aligned}
\end{equation}
Note $\xi^{1,\pm}$, $\xi^{N,\pm}$, $\xi^{N,*}$ and $\xi^{i,i+1}$ all belong to $\bar\Om_N$. 

Fix some $N \ge 2$ and without loss of generality take $\la_0 = 1$. 
Let $(\alpha,\beta,\gamma,\delta)$ and $(\alpha,\beta,\gamma,\delta_*)$ be two groups of boundary rates, such that $\delta(s) \le \delta_*(s)$ for $0 \le s \le t$. 
Define the Markov generator $\bar L_{N,s}$ on $\bar\Om_N$ as 
\begin{align}
\bar L_{N,s} := \bar L_{N,s}^{(1)}+\bar L_{N,s}^{(2)}+\bar L_{N,s}^{(3)}+\bar L_{N,s}^{(4)}, 
\end{align}
where for any $f$ defined on $\bar\Omega_N$, 
\begin{align*}
&\bar L_{N,s}^{(1)}f = \sum_{i=1}^{N-1} \big(p\eta'_i(1-\eta_{i+1})+(1-p)\eta'_{i+1}(1-\eta_i)\big)\big(f(\xi^{i,i+1})-f(\xi)\big) 
\\
&\bar L_{N,s}^{(2)}f = \alpha(s)(1-\eta_1)\big(f(\xi^{1,+})-f(\xi)\big) + \gamma(s)\eta'_1\big(f(\xi^{1,-})-f(\xi)\big) 
\\
&\bar L_{N,s}^{(3)}f = \delta(s)(1-\eta_N)\big(f(\xi^{N,+})-f(\xi)\big) + \beta(s)\eta'_N\big(f(\xi^{N,-})-f(\xi)\big) 
\\
&\bar L_{N,s}^{(4)}f = \big(\delta_*(s)-\delta(s)\big)(1+\eta_N-\eta'_N)\big(f(\xi^{N,*})-f(\xi)\big). 
\label{eq:cp4}
\end{align*}

Denote by $\xi = \xi(s)$ the Markov process generated by $\bar L_{N,s}$. 
Observe that $\xi$ couples the processes associated respectively to $(\alpha,\beta,\gamma,\delta)$ and $(\alpha,\beta,\gamma,\delta_*)$. 
Indeed, if $f$ is a function on $\bar\Om_N$ such that $f(\eta\oplus\eta')=g(\eta)$, 
it is not hard to verify that $\bar L_{N,t}f(\eta\oplus\eta') = L_{N,t}g(\eta)$. 
Similarly, $\bar L_{N,t}f(\eta\oplus\eta') = L'_{N,t}g(\eta')$ if $f(\eta\oplus\eta') = g(\eta')$. 

\begin{proof}[Proof of Lemma \ref{lem:mono-nu}]
We prove here for $(\alpha,\beta,\gamma)=(\alpha_*,\beta_*,\gamma_*)$, $\delta\le\delta_*$. 
The other cases are similar. 
In the coupled process $\xi = \eta\oplus\eta'$, $\eta_i(t) \le \eta'_i(t)$, so that pointwisely, 
\begin{align}
\iint_{\Sigma_T} f(x,t)g(\rho_N(x,t))dx\,dt \le \iint_{\Sigma_T} f(x,t)g(\rho'_N(x,t))dx\,dt 
\end{align}
for positive function $f \in \mathcal C(\Sigma_T)$ and increasing function $g \in \mathcal C([0,1])$. 
From \eqref{eq:33},
\begin{align*}
E^{\mathfrak Q} \left[\iint_{\Sigma_T} f(x,t)dx\,dt\int_0^1 gd\nu_{x,t}\right] \le E^{\mathfrak Q_*} \left[\iint_{\Sigma_T} f(x,t)dx\,dt\int_0^1 gd\nu_{x,t}\right]. 
\end{align*}
As $f$ is an arbitrary continuous positive function, 
\begin{align}
E^{\mathfrak Q} \left[\int_0^1 gd\nu_{x,t}\right] \le E^{\mathfrak Q_*} \left[\int_0^1 gd\nu_{x,t}\right], \quad (x,t)-\text{a.e. in } \Sigma_T.
\end{align}
The conclusion follows since we can approximate the indicator function $\1_{[y,1]}$ by a sequence of continuous increasing functions. 
\end{proof}

\begin{proof}[Proof of Lemma \ref{lem:mono-j}]
We prove here \eqref{eq:j2} with $(\alpha,\beta,\gamma)=(\alpha_*,\beta_*,\gamma_*)$, $\delta\le\delta_*$. 
For the coupled process $\xi = \eta\oplus\eta'$, let $\eta_i^\Delta = \eta'_i-\eta_i$ be the \emph{second class particle process}. 
Recall the counting process $h=h_+-h_-$ defined for $\eta(\cdot)$ in \eqref{eq:counting1}, \eqref{eq:counting2}. 
Define similar counting processes $h'$, $h'_\pm$ and $h^\Delta$, $h_\pm^\Delta$ for $\eta'(\cdot)$ and $\eta^\Delta(\cdot)$, respectively. 
The definition of $\xi$ assures that
\begin{align}
\label{eq:second-par}
h'(i,T) - h(i,T) = h^\Delta(i,T), \quad \forall\,0 \le i \le N. 
\end{align}
Observe that in this case, the particle in $\eta^\Delta$ can enter only from $N$, so that
\begin{align}
\sum_{i=0}^N h^\Delta(i,T) = F\big(\eta^\Delta(0)\big) - F\big(\eta^\Delta(t)\big), \quad F(\eta):= \sum_{j:\,\eta_j=1} (N+1-j).
\end{align}
For any $\eta \in \Om_N$, let $\xi(0) = \eta\oplus\eta$, then $F(\eta^\Delta(0))=0$, hence,
\begin{align}
\sum_{i=0}^N h'(i,T) - \sum_{i=0}^N h(i,T) = \sum_{i=0}^N h^\Delta(i,T) \le 0.
\end{align}
From \eqref{eq:5-2} and \eqref{eq:5-0},
\begin{equation}
  \begin{aligned}
    \mathcal J(T) &= \lim_{N\to\infty} \frac1{N^{2+a}} \sum_{i=0}^N \E_N [h(i,T)] \\
    &\ge \lim_{N\to\infty} \frac1{N^{2+a}} \sum_{i=0}^N \E_N [h'(i,t)] = \mathcal J_*(T).
  \end{aligned}
\end{equation}
The other cases follow from similar arguments. 
\end{proof}


\appendix
\section{Logarithmic Sobolev inequalities}
\label{sec:log-sobol-ineq}

In this appendix we fix a box of length $k$. 
For $\rho \in (0,1)$, let $\nu_\rho$ be the product Bernoulli measure 
on $\Om_k=\{0,1\}^K$ with density $\rho$. 
For $h = 0, 1, \dots, k$, let $\nu_\rho(\eta|h) =\tilde\nu(\eta|h)$ be 
the uniform distribution on 
\begin{align}
\Om_{k,h} := \left\{ \eta\in\Om_k~\bigg|~\sum_{i=1}^k \eta_j =h \right\}, 
\end{align}
and $\bar\nu_\rho(h)$ be the Binomial distribution $\mathcal B(k,\rho)$. 

The log-Sobolev inequality for the simple exclusion (\cite{Yau}) 
yields that 
there exists a universal constant $C_\LS$ such that 
\begin{equation}
\label{eq:logsob}
\begin{aligned}
\sum_{\eta\in\Om_{k,h}} f(\eta)\log f(\eta) \tilde\nu(\eta|h)
\le \frac{C_\LS k^2}2 \sum_{\eta\in\Om_{k,h}}
\sum_{i=1}^{k-1} \left( \sqrt f (\eta^{i,i+1}) - \sqrt f (\eta) \right)^2 \tilde\nu(\eta|h). 
\end{aligned}
\end{equation}
for any $f\ge0$ on $\Om_{k,h}$ such that 
$\sum_{\eta\in\Om_{k,h}} f\tilde \nu(\eta|h) = 1$.. 

In the following we extend \eqref{eq:logsob} to a log-Sobolev inequality 
associated to the product measure $\nu_\rho$ with boundaries. 
The result is necessary for the boundary block estimates in Section \ref{subsec:proof-li} and \ref{subsec:proof-ge}. 

\begin{prop}\label{prop:logsob-bound}
There exists a constant $C_\rho$ such that 
\begin{equation}
\begin{aligned}
\sum_{\eta\in\Om_k} f(\eta)\log f(\eta) \nu_\rho(\eta)
\le C_\rho k^2 \sum_{\eta\in\Om_k} 
\sum_{i=1}^{k-1} \left( \sqrt f (\eta^{i,i+1}) - \sqrt f (\eta) \right)^2 \nu_\rho(\eta) \\
+ C_\rho k \sum_{\eta\in\Om_k} 
\rho^{1-\eta_1}(1-\rho)^{\eta_1} \left( \sqrt f (\eta^1) - \sqrt f (\eta) \right)^2 \nu_\rho(\eta). 
\end{aligned}
\end{equation}
for any $f\ge0$ on $\Om_k$ such that
$\sum_{\eta\in\Om_k} f\nu_\rho = 1$. 
\end{prop}

\begin{proof}
As the reference measures $\nu_\rho$ are equivalent for $0<\rho<1$, without loss of generality we can fix $\rho=1/2$ and thus $\nu_\rho\equiv2^{-k}$. 
Consider the log-Sobolev inequality for the dynamics where particles 
are created and destroyed at each site with intensity $1/2$. 
Since this is a product dynamics, the log-Sobolev constant is uniform in $k$:
\begin{equation}
\label{eq:13}
\begin{aligned}
\frac1{2^k}\sum_{\eta\in\Om_k} f(\eta)\log f(\eta)
\le \frac C{2^{k+1}} \sum_{i=1}^k \sum_{\eta\in\Om_k} 
\left[ \sqrt f(\eta^i) - \sqrt f(\eta) \right]^2. 
\end{aligned}
\end{equation}

We apply a telescopic argument on \eqref{eq:13}. 
For $\eta\in\Omega_k$ and $1 \le i \le k$, let 
\begin{equation}
\begin{aligned}
\tau_0 := \eta, \quad \tau_j := 
\begin{cases}
(\tau_{j-1})^{i-j,i-j+1}, &1 \le j \le i-1 \\
(\tau_{j-1})^1, &j=i \\
(\tau_{j-1})^{j-i,j-i+1}, &i+1 \le j \le 2i-1. 
\end{cases}
\end{aligned}
\end{equation}
Observing that $\tau_{2i-1}=\eta^i$, therefore 
\begin{align}
\sqrt f(\eta^i) - \sqrt f(\eta) = \sum_{j=0}^{2i-2} \left[ \sqrt f(\tau_{j+1}) - \sqrt f(\tau_j) \right], 
\end{align}
and elementary computation then gives 
\begin{align*}
\left[ \sqrt f(\eta^i) - \sqrt f(\eta)\right]^2 \le
&\;4 (i-1) \sum_{0 \le j \le 2(i-1), j\not=i-1} \left[ \sqrt f(\tau_{j+1}) - \sqrt f(\tau_j) \right]^2 \\
&+2 \left[ \sqrt f(\tau_i) - \sqrt f(\tau_{i-1}) \right]^2. 
\end{align*}
Noting that as $\rho=1/2$, $\nu_\rho$ is invariant with respect to the exchange, creation as well as elimination of particles, we obtain by summing up in $\eta$ that 
\begin{equation*}
\begin{split}
\sum_{\eta\in\Om_k} \left[ \sqrt f(\eta^i) - \sqrt f(\eta) \right]^2\nu_\rho(\eta) \le
&\;8(i-1) \sum_{\eta\in\Om_k} \sum_{j=1}^{i-1} \left[ \sqrt f(\eta^{j,j+1}) - \sqrt f(\eta) \right]^2\nu_\rho(\eta) \\
&+ 2\sum_{\eta\in\Om_k} \left[ \sqrt f(\eta^1) - \sqrt f(\eta) \right]^2\nu_\rho(\eta). 
\end{split}
\end{equation*}
Summing up in $i$ we get the required inequality.
\end{proof}


\addcontentsline{toc}{chapter}{References}

\begin{thebibliography}{10}

\bibitem{Baha12} C. Bahadoran,
\emph{Hydrodynamics and Hydrostatics for a
Class of Asymmetric Particle Systems
with Open Boundaries}, Commun. Math. Phys. 310, 1-24 (2012),
(DOI) 10.1007/s00220-011-1395-6

\bibitem{BNL} Bardos, C., Leroux, A.Y, N\'ed\'elec, J.C.:
\emph{First order quasilinear equations with boundary conditions.}
Comm. Part. Diff. Equ. 4, 1017--1034 (1979)
\bibitem{Brak06} R.~Brak, S.~Corteel, J.~Essam, R.~Parviainen, A.~Rechnitzer:
\emph{Combinatorial derivation of the PASEP stationary state}, 
The Electronic Journal of Combinatiorics, 13, R108, (2006),
https://doi.org/10.37236/1134

\bibitem{DOQS} Anna De Masi, Stefano Olla,
\emph{Quasi-static Hydrodynamic limits},
{J. Stat Phys.}, \textbf{161}:1037--1058, (2015), 
https://doi.org/10.1007/s10955-015-1383-x.

\bibitem{DO} Anna De Masi, Stefano Olla, 
\emph{Quasi Static Large Deviations},
{Annales H. Poincare, Probabilit\'es et Statistiques},
Vol. 56, No. 1, 524--542, 2020,
https://doi.org/10.1214/19-AIHP971

\bibitem{Fritz04} J\'ozsef Fritz, \emph{Entropy Pairs and
Compensated Compactness for Weakly Asymmetric Systems},
Advanced Studies in Pure Mathematics 39, 2004
Stochastic Analysis on Large Scale Interacting Systems pp. 143--171.

\bibitem{FT04} J\'ozsef Fritz, B\'alint T\'oth, \emph{Derivation of the
Leroux system as the hydrodynamic limit of a two-component lattice gas},
Communications in Mathematical Physics, 249(1):1--27, Jul 2004.

\bibitem{mox21} Stefano Marchesani, Stefano Olla, Lu Xu,
\emph{Quasi-static limit for a hyperbolic conservation law},
Nonlinear Differential Equations and Applications NoDEA, 28(53): 1--12, 2021, http://doi.org/10.1007/s00030-021-00716-5
\bibitem{Derrida93} Derrida, B., Evans, M. R., Hakim, V., Pasquier, V.:
\emph{ Exact solution of a 1D asymmetric exclusion model
using a matrix formulation}. J. Phys. A 26, 1493--1517 (1993)

\bibitem{Liggett75} Thomas M. Liggett,
\emph{Ergodic Theorems for the Asymmetric Simple Exclusion Process},
Transactions of the American Mathematical Society, Vol. 213 (Nov., 1975),
237--261.

\bibitem{Otto96} F Otto. \emph{Initial-boundary value problem for a
scalar conservation law.}
Comptes rendus de l'Acad\'emie des Sciences. S\'erie 1, Math\'ematique,
322:729--734, 1996.

\bibitem{reza} Rezakhanlou, F.:
\emph{Hydrodynamic limit for attractive particle systems on $Z^d$.}
Commun. Math. Phys.140, 417--448 (1991)
\bibitem{Schutz99} Popkov, V., Sch\"utz, G.: 
\emph{Steady state selection in driven diffusive systems with open boundaries.}
Europhys. Lett. 48, 257--263 (1999)
\bibitem{Uchi04}
Masaru Uchiyama, Tomohiro Sasamoto and Miki Wadati,
\emph{Asymmetric simple exclusion process with open boundaries and Askey--Wilson polynomials},
2004 J. Phys. A: Math. Gen. 37 49--85

\bibitem{lu21}
Lu Xu, \emph{Hydrodynamic limit for asymmetric simple exclusion with accelerated boundaries},
\texttt{arXiv:2103.08019}, 2021
\bibitem{Yau}
H.T. Yau,
\emph{Logarithmic Sobolev inequality for
generalized simple exclusion processes},
Probab. Theory Relat. Fields 109, 507--538 (1997)

\end{thebibliography}
\end{document}